\documentclass[smallextended,final,numbook]{svjour3}     
\usepackage[backref,colorlinks=true, linkcolor=blue, citecolor=blue, urlcolor=blue, pdfborder={0 0 0}]{hyperref}
\usepackage[letterpaper,top=2in, bottom=1.5in, left=1in, right=1in]{geometry}

\smartqed  

\usepackage{tikz}
\usetikzlibrary{fit}\usetikzlibrary{arrows,shapes,positioning}
\usetikzlibrary{decorations.markings}
\tikzstyle arrowstyle=[scale=1]
\tikzstyle directed=[postaction={decorate,decoration={markings,
    mark=at position .5 with {\arrow[arrowstyle]{stealth}}}}]
\tikzstyle reverse directed=[postaction={decorate,decoration={markings,
    mark=at position .5 with {\arrowreversed[arrowstyle]{stealth};}}}]

\spnewtheorem{thm}{Theorem}[section]{\bf}{\it}
\spnewtheorem{cor}[thm]{Corollary}{\bf}{\it}
\spnewtheorem{lem}[thm]{Lemma}{\bf}{\it}
\spnewtheorem{prop}[proposition]{Proposition}{\bf}{\it}
\spnewtheorem{defi}[thm]{Definition}{\bf}{\it}
\spnewtheorem{que}[thm]{Question}{\bf}{\it}
\spnewtheorem{rem}[thm]{Remark}{\bf}{\it}
\spnewtheorem{prob}[thm]{Problem}{\bf}{\it}
\spnewtheorem{fact}[thm]{Fact}{\bf}{\it}
\spnewtheorem{assump}{Assumption}{\bf}{\it}
\usepackage[normalem]{ulem}

\usepackage{mathtools}

\usepackage[normalem]{ulem}

\newcommand{\vx}{{\mathbf{x}}}

\newcommand{\vz}{{\mathbf{z}}}

\newcommand{\vR}{{\mathbf{R}}}

\newcommand{\cC}{{\mathcal{C}}}

\newcommand{\cG}{{\mathcal{G}}}
\newcommand{\cH}{{\mathcal{H}}}

\newcommand{\cK}{{\mathcal{K}}}
\newcommand{\cL}{{\mathcal{L}}}

\newcommand{\cR}{{\mathcal{R}}}

\newcommand{\St}{{\mathrm{subject~to}}} 
\newcommand{\dom}{{\mathrm{dom}}} 
\newcommand{\range}{{\mathrm{range}}} 

\newcommand{\prox}{\mathbf{prox}}

\newcommand{\tnabla}{\widetilde{\nabla}}

\newcommand{\TPRS}{T_{\mathrm{PRS}}}
\newcommand{\best}{\mathrm{best}}
\newcommand{\kbest}{k_{\best}}

\DeclareMathOperator*{\argmin}{arg\,min}

\DeclareMathOperator*{\Min}{minimize}
\DeclareMathOperator*{\Max}{maximize}

\DeclareMathOperator*{\zer}{zer}

\newcommand{\bc}{\begin{center}}
\newcommand{\ec}{\end{center}}

\newcommand{\bdm}{\begin{displaymath}}
\newcommand{\edm}{\end{displaymath}}

\newcommand{\beq}{\begin{equation}}
\newcommand{\eeq}{\end{equation}}

\newcommand{\bfl}{\begin{flushleft}}
\newcommand{\efl}{\end{flushleft}}

\newcommand{\bt}{\begin{tabbing}}
\newcommand{\et}{\end{tabbing}}

\newcommand{\beqn}{\begin{align}}
\newcommand{\eeqn}{\end{align}}

\newcommand{\beqs}{\begin{align*}} 
\newcommand{\eeqs}{\end{align*}}  

\usepackage{multirow} 
\usepackage{arydshln} 

\usepackage{mathtools}
\usepackage{algorithm}
\usepackage{algpseudocode}

\newcommand\numberthis{\addtocounter{equation}{1}\tag{\theequation}}

\DeclarePairedDelimiter{\dotp}{\langle}{\rangle}
 \DeclarePairedDelimiter{\ceil}{\lceil}{\rceil}

\newcommand{\refl}{\mathbf{refl}}

\usepackage{booktabs}
\usepackage{multirow} 
\usepackage{arydshln} 
\def\cut#1{{}}

\begin{document}

\title{Faster convergence rates of relaxed Peaceman-Rachford and ADMM under regularity assumptions}
\author{Damek Davis   \and
        Wotao Yin}

\institute{D. Davis \and W. Yin\at
              Department of Mathematics, University of California, Los Angeles\\
              Los Angeles, CA 90025, USA\\
              \email{damek / wotaoyin@ucla.edu}           
}

\date{Received: date / Accepted: date}

\maketitle

\abstract{Splitting schemes are a class of powerful algorithms that solve complicated monotone inclusion and convex optimization problems that are built from many simpler pieces. They give rise to algorithms in which the simple pieces of the decomposition are processed individually. This leads to easily implementable and highly parallelizable algorithms, which often obtain nearly state-of-the-art performance.

In this paper, we provide a comprehensive convergence rate analysis of the Douglas-Rachford splitting (DRS), Peaceman-Rachford splitting (PRS), and alternating direction method of multipliers (ADMM) algorithms under various regularity assumptions including strong convexity, Lipschitz differentiability, and bounded linear regularity.  The main consequence of this work is that relaxed PRS and ADMM automatically adapt to the regularity of the problem and achieve convergence rates that improve upon the (tight) worst-case rates that hold in the absence of such regularity. All of the results are obtained using simple techniques.}


\section{Introduction}

The Douglas-Rachford splitting (DRS), Peaceman-Rachford splitting (PRS), and alternating direction method of multipliers (ADMM) algorithms are abstract splitting schemes that solve monotone inclusion and convex optimization problems \cite{lions1979splitting,GlowinskiADMM,gabay1976dual}.  The DRS and PRS algorithms solve monotone inclusion problems in which the operator is the sum of two (possibly) simpler operators by accessing each operator individually through its resolvent.  The ADMM algorithm solves convex optimization problems in which the objective is the sum of two (possibly) simpler functions with variables linked through a linear constraint via an alternating minimization strategy.  The variable splitting that occurs in each of these algorithms can give rise to parallel and even distributed implementations of minimization algorithms \cite{boyd2011distributed,shi2013linear,wei2012distributed}, which are particularly suitable for large-scale applications.
Since the 1950s, these methods were largely applied to solving partial differential equations (PDEs) and feasibility problems, and only recently has their power been utilized in (PDE and non-PDE related) image processing, statistical and machine learning, compressive sensing, matrix completion,  finance, and control \cite{goldstein2009split,boyd2011distributed}.

In this paper, we consider two prototype optimization problems: the unconstrained problem
\begin{align}\label{eq:simplesplit}
\Min_{x\in \cH}~ f(x) + g(x)
\end{align}
where $\cH$ is a Hilbert space, and the linearly constrained variant
\begin{align*}
\Min_{x \in \cH_1,~ y \in \cH_2} & \; f(x) + g(y) \\
\St~ &  \; Ax + By = b \numberthis \label{eq:simplelinearconstrained}
\end{align*}
where $\cH_1, \cH_2$, and $\cG$ are Hilbert spaces, the vector $b$ is an element of $\cG$, and $A : \cH_1 \rightarrow \cG$ and $B : \cH_2 \rightarrow \cG$ are linear operators. 
Problem~\eqref{eq:simplesplit} models a variety of tasks in signal recovery where one function corresponds to a \emph{data fitting term} and the other enforces \emph{prior knowledge}, such as sparsity, low rank, or smoothness \cite{combettes2011proximal}.  In this paper, we apply relaxed PRS (Algorithm~\ref{alg:DRS}) to solve Problem~\eqref{eq:simplesplit}.  On the other hand, Problem~\eqref{eq:simplelinearconstrained} models tasks in machine learning, image processing and distributed optimization. The linear constraint can be used to enforce data fitting, but it can also be used to split variables in a way that gives rise to parallel or distributed optimization algorithms \cite{bertsekas1989parallel,boyd2011distributed}.  We will apply relaxed ADMM (Algorithm~\ref{alg:ADMM}) to Problem~\eqref{eq:simplelinearconstrained}.

\subsection{Goals, challenges, and approaches}
This work improves the theoretical understanding of DRS, PRS, and ADMM, as well as their averaged versions. When applied to convex optimization problems, they are known to converge under rather general conditions \cite[Corollary 27.4]{bauschke2011convex}. 
This work seeks to complement the results of \cite{davis2014convergence}, which are developed under general convexity assumptions, by deriving stronger rates under correspondingly stronger conditions on Problems~\ref{eq:simplesplit} and~\ref{eq:simplelinearconstrained}. One of the main consequences of this work is that the relaxed PRS and ADMM algorithms automatically adapt to the regularity of the problem at hand and achieve convergence rates that improve upon the worst-case rates shown in \cite{davis2014convergence} for the nonsmooth case. Thus, our results offer an explanation of the great performance of relaxed PRS and ADMM observed in practice, and together with \cite{davis2014convergence} we now have a comprehensive convergence rate analysis of the relaxed PRS and ADMM algorithms.
\begin{table}
\renewcommand{\arraystretch}{1.3}
\begin{center}
\begin{tabular}{c|c|c|c}\hline
Regularity assumption & \multicolumn{2}{|c|}{Objective error}  & \multirow{2}{*}{FPR} \\\cline{2-3}
beyond convexity & Rate & Type  &\\\hline\hline
None & \multicolumn{2}{|c|}{not available}  & \multirow{5}{*}{$o(1/k)$} \\\cline{1-3}
\multirow{2}{*}{Lipschitz $f$ or $g$~\cite{davis2014convergence}} & $o(1/\sqrt{k})$ & nonergodic & \\\cdashline{2-3}
 & $O(1/k)$ & ergodic$^\dag$  & \\\cline{1-1}\cline{2-3}
\multirow{2}{*}{Strongly convex $f$ or $g$} & $o(1/k)$ & best itr. &  \\\cdashline{2-3}
 & $O(1/k)$ & ergodic$^\dag$ &  \\\cline{1-3}
\multirow{2}{*}{Lipschitz  $\nabla g$} & $o(1/k)$ &  best itr.  & \\\cline{2-4}
& $o(1/k)$$^\ddag$ & nonergodic$^\ddag$ & $o(1/k^2)$$^\ddag$  \\\hline
Lipschitz $\nabla f$ or $\nabla g$,  & \multirow{2}{*}{$O(e^{-k})$} & \multirow{2}{*}{R-linear} & \multirow{2}{*}{$O(e^{-k})$}\\
strongly convex $f$ or $g$ & & & \\\hline
$f = d_{C_1}^2$ and $g = d_{C_2}^2$,  & \multirow{2}{*}{$O(e^{-k})$} & \multirow{2}{*}{R-linear} & \multirow{2}{*}{$O(e^{-k})$}\\
$\{C_1, C_2\}$ linearly regular & & & \\\hline
\end{tabular}\\[5pt]

\caption{Summary of convergence rates for \emph{relaxed PRS} with relaxation parameters $\lambda_k\in(\epsilon,1-\epsilon)$, for any $\epsilon>0$. FPR stands for  the fixed-point residual $\|\TPRS z^{k} - z^k\|^2$. $^\dag$These two ergodic rates hold for $\lambda_k\in(\epsilon,1]$. $^\ddag$These rates hold for DRS ($\lambda_k=1/2$) and properly bounded  step size $\gamma$.\label{tb:rPRS}}
\end{center}
\end{table}

\begin{table}
\begin{center}
\begin{tabular}{l|c|c|c|c|c}\hline
&\multicolumn{3}{c|}{Regularity assumption beyond convexity} & Convergence & \multirow{2}{*}{Type} \\\cline{2-4}
& Strongly convex & Lipschitz & Full  rank & rate &  \\\hline\hline
\multirow{4}{*}{1} & \multirow{4}{*}{-} & \multirow{4}{*}{-} & \multirow{4}{*}{-} & $o(1/k)$ & nonergodic feas. \\\cline{5-6}
& &  &  & $O(1/k^2)$ & ergodic feas. \\\cline{5-6}
& &  &  & $o(1/\sqrt{k})$ & nonergodic obj. error \\\cline{5-6}
& &  &  & $O(1/k)$ & ergodic obj. error\\\hline
\multirow{2}{*}{2} & \multirow{2}{*}{$g$} & \multirow{2}{*}{-} & \multirow{2}{*}{-} & $o(1/k^2)$ & feasibility \\\cline{5-6}
& &  &  & $o(1/k)$ & objective \\\hline
3 & $g$ & $\nabla g$ & $B$ (row rank) & \multirow{4}{*}{$O(e^{-k})$} & R-linear \\\cline{1-4}
4 & $f$ & $\nabla f$ & $A$ (row rank) &  & feasibility, \\\cline{1-4}
5 & $f$ & $\nabla g$ & $B$ (row rank) &  & objective error, \\\cline{1-4}
6 & $g$ & $\nabla f$ & $A$ (row rank) &  & solution error \\
\hline
\end{tabular}\\[5pt]
\end{center}
\caption{Summary of convergence rates for relaxed ADMM. Feasibility is $\|Ax^k+By^k-b\|^2$,  objective error is $(f(x^k)+g(y^k))-(f(x^*)+g(y^*))$, and solution error includes $\|w^k-w^*\|^2$, $\|Ax^k-Ax^*\|^2$, and $\|By^k-By^*\|^2$. Case 1 is from~\cite{davis2014convergence}, where the nonergodic rates  hold for relaxation parameters $\lambda_k\in(\epsilon,1-\epsilon)$, for any $\epsilon>0$, and the ergodic rates hold for $\lambda_k\in(\epsilon,1]$. Case 2 also requires a  bounded step size. Each of cases 3--6 ensures R-linear convergence.  
\label{tb:ADMM}
}
\end{table}

In this paper, we derive the convergence rates of the objective error and fixed-point residual (FPR) of  relaxed PRS applied to Problem~\eqref{eq:simplesplit}; see Table \ref{tb:rPRS}. In addition, we derive the convergence rates of the constraint violations and objective errors 
for relaxed ADMM applied to Problem~\eqref{eq:simplelinearconstrained}; see Table \ref{tb:ADMM}. {By appealing to counterexamples in~\cite{davis2014convergence}, several of the rates in Table~\ref{tb:rPRS} can be shown to be tight up to constant factors.} 

The derived rates are useful for determining how many iterations of the relaxed PRS and ADMM algorithms are needed in order to reach a certain accuracy,  to decide when to stop an algorithm, and to compare relaxed PRS and ADMM to other algorithms in terms of their worst-case complexities.

\subsection{Notation}
In what follows,   $\cH, \cH_1, \cH_2, \cG$ denote (possibly infinite dimensional) Hilbert spaces. In fixed-point iterations, $(\lambda_j)_{j \geq 0} \subset \vR_+$ will denote a sequence of relaxation parameters, and
\begin{equation}\label{def:Lambda}\Lambda_k := \sum_{i=0}^k \lambda_i
\end{equation} is its $k$th partial sum. To ease notational memory, the reader may assume that $\lambda_k \equiv (1/2)$ and $\Lambda_k =(k+1)/2$ in the DRS algorithm, or that $\lambda_k \equiv 1$ and $\Lambda_k = (k+1)$ in the PRS algorithm. Given the sequence  $(x^j)_{j \geq 0}\subset \cH$, we let $\overline{x}^k = ({1}/{\Lambda_k})\sum_{i=0}^k \lambda_i x^i$ denote its $k$th average  with respect to the sequence $(\lambda_j)_{j \geq 0}$.
A convergence result is \emph{ergodic} if it applies to the sequence $(\overline{x}^j)_{j \geq 0}$, and \emph{nonergodic} if it applies to the sequence $(x^j)_{j \geq 0}$.

Given a closed, proper, and convex function $f : \cH \rightarrow (-\infty, \infty]$, the set $\partial f(x)$ denotes its subdifferential at $x$ and
$
\tnabla f(x) \in \partial f(x)
$
denotes a subgradient. (This notation was used in \cite[Eq. (1.10)]{bertsekas2011incremental}.) The convex conjugate of a closed, proper, and convex function $f$ is
$
f^\ast(y) := \sup_{x \in \cH} \dotp{y, x} - f(x).
$
Let $I_{\cH}: \cH \rightarrow \cH$ denote the identity map. For any point $x \in \cH$ and  $\gamma \in \vR_{++}$, we let
$\prox_{\gamma f}(x) := \argmin_{y \in \cH} f(y) + \frac{1}{2\gamma} \|y - x\|^2$ and $\refl_{\gamma f} := 2\prox_{\gamma f} - I_{\cH},$
which are known as the \emph{proximal} and \emph{reflection} operators. In addition, we define the PRS operator:
\begin{align*}
\TPRS &:= \refl_{\gamma f} \circ \refl_{\gamma g}.
\end{align*}
Let $\lambda > 0$. For every nonexpansive map $T : \cH \rightarrow \cH$ we define the averaged map:
\begin{align*}
T_{\lambda} := (1-\lambda)I_{\cH} + \lambda T.
\end{align*}
We call the following identity the \emph{cosine rule}:
\begin{align*}
\|y-z\|^2+2\dotp{y-x,z-x}=\|y-x\|^2+\|z-x\|^2,\quad\forall x,y,z\in\cH \numberthis\label{eq:cosinerule}.
\end{align*}
\subsection{Assumptions}

We list the the assumptions used throughout this papers as follows.

\begin{assump}[Problem assumptions]
Every function we consider is closed, proper, and convex.
\end{assump}
Unless otherwise stated, a function is not necessarily differentiable.

\begin{assump}[Solution existence]\label{assump:additivesub}
Functions $f, g : \cH \rightarrow (-\infty, \infty]$ satisfy
\begin{align}
\zer(\partial f + \partial g) \neq \emptyset.
\end{align}
\end{assump}
Note that this assumption is slightly stronger than the existence of a minimizer because $\zer(\partial f + \partial g) \neq \zer(\partial (f + g))$, in general \cite[Remark 16.7]{bauschke2011convex}. Nevertheless, this assumption is standard.

\begin{assump}[Differentiability]
Every differentiable function is Fr{\'e}chet differentiable \cite[Def. 2.45]{bauschke2011convex}.
\end{assump}

\subsection{The Douglas-Rachford and relaxed Peaceman-Rachford Splitting Algorithms}

The results of this paper apply to several operator-splitting algorithms that are all based on the atomic evaluation of the \emph{proximal operator}. By default, all algorithms start from an arbitrary $z^0 \in \cH$. The Douglas-Rachford splitting (DRS) algorithm applied to minimizing $f+g$ is as follows:
\begin{align*}
\begin{cases}
x_g^k = \prox_{\gamma g}(z^k);\\
x_f^k = \prox_{\gamma f}( 2 x_g^k - z^k);\\
z^{k+1} = z^k +  (x_f^k - x_g^k);
\end{cases}
\quad k = 0, 1, \ldots,
\end{align*}
which has the equivalent operator-theoretic and subgradient form (Lemma~\ref{prop:DRSmainidentity}):
\begin{align*}
z^{k+1} &= \frac{1}{2}(I_{\cH} + \TPRS)(z^k) = z^k - \gamma(\tnabla f(x_f^k) + \tnabla g(x_g^k)), \quad k = 0, 1, \ldots,
\end{align*}
where $\tnabla f(x_f^k)\in\partial f(x_f^k)$ and  $\tnabla g(x_g^k)\in \partial g(x_g^k)$. (See Part~\ref{prop:basicprox:part:optprox} of Proposition~\ref{prop:basicprox} for how the notation $\tnabla$ relates to $\prox$.) In the above algorithm, we can replace the $(1/2)$-average of $I_{\cH}$ and $\TPRS$ with any other weight; this results the \emph{relaxed PRS} algorithm:

\begin{algorithm}[H]
\begin{algorithmic}
\Require $z^0 \in \cH, ~\gamma > 0, ~(\lambda_j)_{j \geq 0}\subset   (0, 1]$
\For{$k=0,~1,\ldots$}
\State $z^{k+1} = (1-\lambda_k)z^k +   \lambda_k\refl_{\gamma f} \circ \refl_{\gamma g}(z^k) $\;
\EndFor
\end{algorithmic}
\caption{{Relaxed Peaceman-Rachford Splitting (relaxed PRS)}}
\label{alg:DRS}
\end{algorithm}
The special cases $\lambda_k \equiv 1/2$ and $\lambda_k \equiv 1$ are called the DRS and PRS algorithms, respectively.

\subsection{Practical implications: a comparison with forward-backward splitting}\label{section:FBSpractical}

Suppose that the function $g$ in Problem~\ref{eq:simplesplit} is differentiable and $\nabla g$ is $(1/\beta)$-Lipschitz.  Under this smoothness assumption, we can apply FBS algorithm to Problem~\ref{eq:simplesplit}: given $z^0 \in \cH$, for all $k \geq 0$, define
\begin{align*}
z^{k+1} = \prox_{\gamma f} (z^k - \gamma \nabla g(z^k)).
\end{align*}
To ensure convergence, the  stepsize parameter $\gamma$ must be strictly less than $2\beta$.

Now because the gradient operator is often simpler to evaluate than the proximal operator, it may be preferable to use FBS instead of relaxed PRS whenever one of the objectives is differentiable.  From our results, we can give two reasons why it may be preferable to use relaxed PRS over FBS:
\begin{enumerate}
\item If the Lipschitz constant of the gradient is known, our analysis indicates how to properly choose stepsizes of relaxed PRS so that both algorithms converge with the same rate (Theorem~\ref{thm:differentiableobjective}). In practice, relaxed PRS is often observed to converge faster than FBS, so our results at least indicate that we can do no worse by using relaxed PRS.
\item If the Lipschitz constant of the gradient is not known,  a line search procedure can be used to guarantee convergence of FBS.  If this procedure is more expensive than evaluating the proximal operator, then relaxed PRS should be used.  Indeed, Theorem~\ref{thm:lipschitzbest} shows that the ``best iterate" of relaxed PRS will converge with rate $o(1/(k+1))$ regardless of the chosen stepsize, whereas FBS may fail to converge.
\end{enumerate}

Thus, one of our main contributions is the ``demystification" of parameter choices, and a partial explanation of the perceived practical advantage of relaxed PRS over FBS.

\subsection{Basic properties of proximal operators}\label{sec:nonexpansive}
The following properties are included in textbooks such as \cite{bauschke2011convex}.

\begin{proposition}\label{prop:basicprox}
Let $f, g : \cH \rightarrow (-\infty, \infty)$ be closed, proper, and convex functions, and let $T : \cH \rightarrow \cH$ be nonexpansive. The the following are true:
\begin{enumerate}
\item\label{prop:basicprox:part:optprox} {\bf Optimality conditions of $\prox$:} Let $x \in \cH$. Then $x^+ = \prox_{\gamma f} (x)$ if, and only if, $$\tnabla f(x^+) :=\frac{1}{\gamma}(x-x^+) \in \partial f(x^+).$$
\item \label{cor:proxcontraction} {\bf The proximal operator $\prox_{\gamma f} :  \cH \rightarrow \cH$ is ${1}/{2}$-averaged:} 
\begin{align}\label{cor:proxcontraction:eq:main}
\|\prox_{\gamma f}(x) - \prox_{\gamma f}(y) \|^2 &\leq \|x - y\|^2 - \|(x - \prox_{\gamma f}(x)) - (y - \prox_{\gamma f}(y))\|^2.
\end{align}
\item \label{prop:basicprox:part:nonexpansive} {\bf Nonexpansiveness of the PRS operator:} The operator $\refl_{\gamma f} : \cH \rightarrow \cH$ is nonexpansive. Therefore, the composition
$\TPRS = \refl_{\gamma f} \circ \refl_{\gamma g}.$
is nonexpansive.
\end{enumerate}
\end{proposition}

\subsection{Convergence rates of summable sequences}
The following facts will be key to deducing Convergence rates in Sections~\ref{sec:generalstrongconvexity} and~\ref{sec:lipschitzderivatives}.  It originally appeared in \cite[Lemma 3]{davis2014convergence}.

\begin{fact}[Summable sequence convergence rates]\label{lem:sumsequence}
Suppose that the nonnegative scalar sequences  $(\lambda_j)_{j\geq0}$ and $(a_{j})_{j\geq 0}$ satisfy $\sum_{i=0}^\infty \lambda_ia_i < \infty$, and define $\Lambda_k$ as in Equation~\eqref{def:Lambda}.
\begin{enumerate}
\item \label{lem:sumsequence:part:main} \textbf{Monotonicity:} If $(a_j)_{j \geq0}$ is \emph{monotonically nonincreasing}, then
\begin{align}\label{eq:sumsequence-bigo}
a_{k} \leq \frac{1}{\Lambda_k}\left(\sum_{i=0}^\infty\lambda_i a_i \right) && and && a_{k} = o\left(\frac{1}{\Lambda_{k} - \Lambda_{\ceil{{k}/{2}}}}\right).
\end{align}
\item \label{lem:sumsequence:part:b} \textbf{Faster rates:} Suppose $(b_j)_{j\geq0}$ is a nonnegative scalar sequence, that $\sum_{i=0}^\infty b_j < \infty$, and that $\lambda_k a_k \leq b_k - b_{k+1}$ for all $k\geq 0$. Then the following sum is finite:
\begin{align}
\sum_{i=0}^\infty (i+1)\lambda_ia_i \leq \sum_{i=0}^\infty b_i
\end{align}
\item \label{lem:sumsequence:part:nonmono} \textbf{No monotonicity:} For all $k \geq 0$, define the sequence of ``best indices'' with respect to $(a_j)_{j\ge 0}$ as
\begin{align*}
k_{\mathrm{best}} &:= \argmin_i\{a_i | i = 0, \cdots, k\}.
\end{align*}
Then $(a_{j_{\mathrm{best}}})_{j \geq 0}$ is nonincreasing, and the above bounds continue to hold when $a_k$ is replaced with $a_{\kbest}$.  
\end{enumerate}
\end{fact}

\subsection{Convergence of the fixed-point residual (FPR)}\label{sec:FPR}

We will need to following facts in our analysis below: 
\begin{fact}[Convergence rates of FPR]\label{fact:averagedconvergence}
Let $z^\ast \in \cH$ be a fixed point of $\TPRS$, and let $(z^j)_{j \geq 0}$ be generated by the relaxed PRS algorithm: $z^{k+1} = (\TPRS)_{\lambda_k}z^k. $ Then the following are true (\cite[Theorem 1]{davis2014convergence}):
\begin{enumerate}
\item $(\|z^j - z^\ast\|)_{j \geq 0}$ is monotonically nonincreasing; \label{fact:averagedconvergence:eq:mono}
\item 
$(\|\TPRS z^j - z^j\|)_{j \geq 0}$ is monotonically nonincreasing, and thus so is $((1/\lambda_j)\|z^{j+1}-z^j\|)_{j \geq 0}$; \label{fact:averagedconvergence:eq:FPRmono}
\item If $\lambda_k \equiv \lambda$, then $(\|(\TPRS)_\lambda z^j - z^\ast\|)_{j\geq0}$ is monotonically nonincreasing;
\item The Fej\'er-type inequality holds: for all $\lambda \in (0, 1]$
\begin{align*}
\|(\TPRS)_{\lambda}z^k - z^\ast\|^2 &\leq \|z^k - z^\ast\|^2 - \frac{1- \lambda}{\lambda}\|(\TPRS)_{\lambda} z^k - z^k\|^2.\numberthis\label{eq:fejer}
\end{align*}
\item For all $k \geq 0$, let $\tau_k = \lambda_k(1-\lambda_k)$. Then $\sum_{i=0}^\infty \tau_i\|\TPRS z^i - z^i\|^2 \leq \|z^0 - z^\ast\|^2.$
\item If $\underline{\tau} := \inf_{j \geq 0} \lambda_k(1-\lambda_k) > 0$, then the following convergence rates hold:
\begin{align}\label{cor:DRSaveragedconvergence:eq:main}
\|\TPRS z^{k} - z^k\|^2 \leq \frac{\|z^{0} - z^\ast\|^2}{\underline{\tau}(k+1)} && and &&  \|\TPRS z^{k} - z^k\|^2 = o\left(\frac{1}{\underline{\tau}(k+1)}\right).
\end{align}
\end{enumerate}
\end{fact}

\begin{remark}
We call the quantity
$\|\TPRS z^{k} - z^k\|^2$
the fixed-point residual (FPR) of the relaxed PRS algorithm. Throughout this paper, we slightly abuse terminology and call the successive iterate difference $\|z^{k+1} - z^k\|^2 = \lambda_k^2\|\TPRS z^k - z^k\|^2$ FPR as well.
\end{remark}

\subsection{Subgradients}

Lemma~\ref{prop:DRSmainidentity} is key to deducing all of the algebraic relations necessary for relating the objective error to the FPR of the relaxed PRS iteration
\begin{lemma}\label{prop:DRSmainidentity}
Let $z\in \cH$. Define auxiliary points $x_g := \prox_{\gamma g}(z)$ and $x_f := \prox_{\gamma f}(\refl_{\gamma g}(z))$. Then the identities hold:
\begin{align}
x_g = z - \gamma \tnabla g(x_g)  && \mathrm{and} && x_f &= x_g - \gamma \tnabla g(x_g) - \gamma \tnabla f(x_f). \label{prop:DRSmainidentity:f}
\end{align}
In addition, each relaxed PRS step $z^+=(\TPRS )_{\lambda}(z)$ has the following representation:
\begin{align}\label{eq:DRSmainidentity2}
z^+ - z = 2\lambda(x_f- x_g) = -2  \lambda\gamma(\tnabla g(x_g)+\tnabla f(x_f)).
\end{align}
\end{lemma}

\subsection{Fundamental inequalities}

Throughout the rest of the paper we will  use the following notation: Every function $f$ is $\mu_f$-strongly convex and $\tnabla f$ is $(1/\beta_f)$-Lipschitz. Note that if $\beta_f > 0$, then $f$ is differentiable and $\tnabla f = \nabla f$. However, we also allow the strong convexity or Lipschitz differentiability constants to vanish, in which case $\mu_f = 0$ or $\beta_f = 0$ and $f$ may fail to posses either regularity property. Thus, we always have the inequality \cite[Theorem 18.15]{bauschke2011convex}:
\begin{align}\label{eq:strongconvexandlipschitzlowerbound}
f(x) &\geq f(y) + \dotp{x-y, \tnabla f(y)} + S_f(x, y),
\end{align}
where
\begin{align}\label{eq:snotation}
S_f(x, y) &:=  \max\left\{\frac{\mu_f}{2}\|x - y\|^2, \frac{\beta_f}{2}\|\tnabla f(x) - \tnabla f(y)\|^2\right\}.
\end{align}
Note that there is a slight technicality in that $S_f(x, y)$ is only defined where $\partial f(x) \neq \emptyset$.  In particular, we only derive bounds on $S_f(x, y)$ where this is satisfied. 

The following two fundamental inequalities are straightforward modifications of the fundamental inequalities that appeared in \cite[Propositions 4 and 5]{davis2014convergence}.  When these bounds are iteratively applied, they bound the objective error by the sum of a telescoping sequence and a multiple of the FPR.

\begin{proposition}[Upper fundamental inequality]\label{prop:DRSupper}
Let $z \in \cH$,  let $z^+ = (\TPRS)_{\lambda}(z)$, and let $x_f$ and $x_g$ be defined as in Lemma~\ref{prop:DRSmainidentity}. Then for all $x\in \dom(f) \cap \dom(g)$ where $\partial f(x) \neq \emptyset$ and $\partial g(x) \neq \emptyset$, we have
\begin{align*}
4\gamma\lambda\big(f(x_f) + g(x_g) &- f(x) - g(x) + S_f(x_f, x) + S_g(x_g, x)\big) \\
&\leq \|z - x\|^2 - \|z^+ - x\|^2 + \left(1 - \frac{1}{\lambda}\right)\|z^{+} - z\|^2. \numberthis \label{prop:DRSupper:eq:main}
\end{align*}
\end{proposition}
In our analysis below, we will use the upper inequality
\begin{align*}
4\gamma\lambda\big(f(x_f) + g(x_g) &- f(x^*) - g(x^*) + S_f(x_f, x^*) + S_g(x_g, x^*)\big) \\
&\leq \|z - z^*\|^2 - \|z^+ - z^*\|^2 +2\dotp{z-z^+,z^*-x^*}+ \left(1 - \frac{1}{\lambda}\right)\|z^{+} - z\|^2, \numberthis \label{prop:DRSupper:eq:aux}\end{align*}
which is obtained from \eqref{prop:DRSupper:eq:main} by letting $x=x^*$ and applying 
$\|z - x^*\|^2 - \|z^+ - x^*\|^2=\|z-z^*\|^2-\|z^+-z^*\|^2+2\dotp{z-z^+,z^*-x^*}. $

\begin{proposition}[Lower fundamental inequality]\label{prop:DRSlower}
Let $z^\ast$ be a fixed point of $\TPRS $, and let $x^\ast = \prox_{\gamma g}(z^\ast)$.  Then for all $x_f \in  \dom(f)$ and $x_g\in \dom(g)$, the lower bound holds:
\begin{align*}
f(x_f) + g(x_g) - f(x^\ast) - g(x^\ast) &\geq \frac{1}{\gamma }\dotp{x_g - x_f,  z^\ast - x^\ast} + S_f(x_f, x^\ast) + S_g(x_f, x^\ast). \numberthis \label{prop:DRSlower:eq:main}
\end{align*}
\end{proposition}

\section{Strong convexity}\label{sec:generalstrongconvexity}

The following theorem will deduce the convergence of $S_f(x_f^k, x^\ast)$ and $S(x_g^k, x^\ast)$ (see Equation~\eqref{eq:snotation}). In particular, if either $f$ or $g$ is strongly convex and the sequence $(\lambda_j)_{j \geq 0}\subseteq (0, 1]$ is bounded away from zero, then $x_f^k$ and $x_g^k$ converge strongly to a minimizer of $f + g$.  Equation~\eqref{prop:sumauxilliaryterms:eq:main} is the main inequality needed to deduce linear convergence of the relaxed PRS algorithm (Section~\ref{sec:linearconvergence}), and it will reappear several times.

\begin{theorem}[Auxiliary term bound]\label{prop:sumauxilliaryterms}
Suppose that $(z^j)_{j \geq 0}$ is generated by Algorithm~\ref{alg:DRS}.  Then for all $k \geq 0$,
\begin{align}\label{prop:sumauxilliaryterms:eq:main}
8\gamma\lambda_k(S_{f}(x_f^k, x^\ast) + S_g(x_g^k, x^\ast)) &\leq \|z^{k} - z^\ast\|^2 - \|z^{k+1} - z^\ast\|^2 + \left(1 - \frac{1}{\lambda_k}\right)\|z^{k+1} - z^k\|^2.
\end{align}
Therefore, $8\gamma\sum_{i=0}^\infty \lambda_k(S_{f}(x_f^i, x^\ast) + S_g(x_g^i, x^\ast)) \leq \|z^0 - z^\ast\|^2$, and 
\begin{enumerate}
\item \label{prop:sumauxilliaryterms:part:1} \textbf{Best iterate convergence:} If $\underline{\lambda} := \inf_{j \geq 0} \lambda_j > 0$, then 
$\min_{i=0, \cdots, k}\left\{S_{f}(x_f^{i}, x^\ast)\right\} = o\left(1/(k+1)\right)$ and $\min_{i=0, \cdots, k}\left\{S_{g}(x_g^{i}, x^\ast)\right\} = o\left(1/(k+1)\right).$
\item \label{prop:sumauxilliaryterms:part:2} \textbf{Ergodic convergence:} Let $\overline{x}_f^k = (1/\Lambda_k) \sum_{i=0}^k \lambda_ix_f^i$ and $\overline{x}_g^k = (1/\Lambda_k) \sum_{i=0}^k\lambda_i x_g^i$. Then
\begin{align*}
 \overline{S}_{f}(x_f^k, x^\ast)+\overline{S}_{g}(x_g^k, x^\ast)\leq \frac{\|z^0 - z^\ast\|^2}{8 \gamma \Lambda_k}
\end{align*}
where $ \overline{S}_{f}(x_f^k, x^\ast):=\max\bigg\{\frac{\mu_f}{2}\left\|\overline{x}_f^k - x^\ast\right\|^2,~  \frac{\beta_f}{2}\bigg\|\frac{1}{\Lambda_k}\sum_{i=0}^k \tnabla f(x_f^k) - \tnabla f(x^\ast)\bigg\|^2\bigg\}  $
and $\overline{S}_{g}(x_g^k, x^\ast)$ is similarly defined.
\item  \label{prop:sumauxilliaryterms:part:3} \textbf{Nonergodic convergence:} If $\underline{\tau} = \inf_{j \geq 0} \lambda_j(1-\lambda_j)> 0$, then 
$S_{f}(x_f^{k}, x^\ast) + S_g(x_g^{k}, x^\ast) = o\left(1/\sqrt{k+1}\right). $
\end{enumerate}
\end{theorem}
\begin{proof}
By assumption, the relaxation parameters satisfy $\lambda_k \leq 1$. Therefore, Equation~\eqref{prop:sumauxilliaryterms:eq:main} is a consequence of the following inequalities:
\begin{align*}
8\gamma\lambda_k(S_{f}(x_f^k, x^\ast) + S_g(x_g^k, x^\ast)) &\stackrel{\eqref{prop:DRSlower:eq:main},\eqref{eq:DRSmainidentity2}}{\leq} 4\gamma\lambda_k(f(x_f^k) + g(x_g^k) - f(x^\ast) - g(x^\ast) + S_{f}(x_f^k, x^\ast) + S_g(x_g^k, x^\ast))  \\
&\hspace{30pt}- 2\dotp{z^{k} - z^{k+1}, z^\ast - x^\ast} \\
&\stackrel{\eqref{prop:DRSupper:eq:aux}}{\leq}  \|z^{k} - z^\ast\|^2 - \|z^{k+1} - z^\ast\|^2 + \left(1 - \frac{1}{\lambda_k}\right)\|z^{k+1} - z^k\|^2 \\
&\leq  \|z^{k} - z^\ast\|^2 - \|z^{k+1} - z^\ast\|^2.  \numberthis \label{eq:upperboundonstrongterms}
\end{align*}
Note that the sum of Equation~\eqref{eq:upperboundonstrongterms} over all $k$ is indeed bounded by $\|z^0 - z^\ast\|^2$. Thus, Part~\ref{prop:sumauxilliaryterms:part:1} follows from Fact~\ref{lem:sumsequence}, and Part~\ref{prop:sumauxilliaryterms:part:2} follows from Jensen's inequality applied to $\|\cdot \|^2.$

Fix $k \geq 0$, let $z_\lambda = (\TPRS)_{\lambda}z^k$ for $\lambda \in [0, 1]$, and note that $z_{\lambda} - z^k = \lambda( \TPRS z^k - z^k)$. Fact~\ref{fact:averagedconvergence} shows that $\|z_{\lambda} - z^\ast\| \leq \|z^k - z^\ast\|$ (Equation~\eqref{eq:fejer}) and that the sequence $(\|z^j - z^\ast\|)_{j \geq 0}$ is nonincreasing. Therefore, Part~\ref{prop:sumauxilliaryterms:part:3} is a consequence of the cosine rule, Fact~\ref{fact:averagedconvergence}, Equation~\eqref{prop:sumauxilliaryterms:eq:main}, and the following inequalities:
\begin{align*}
S_{f}(x_f^{k}, x^\ast) + S_g(x_g^{k}, x^\ast) &\stackrel{\eqref{prop:sumauxilliaryterms:eq:main}}{\leq} \inf_{\lambda \in [0, 1]} \frac{1}{8\gamma \lambda}\left( \|z^{k} - z^\ast\|^2 - \|z_\lambda - z^\ast\|^2 + \left(1 - \frac{1}{\lambda}\right)\|z^{k} - z_\lambda\|^2\right)\\
&\stackrel{\eqref{eq:cosinerule}}{=} \inf_{\lambda \in [0,1]}\frac{1}{8\gamma \lambda}\left(2\dotp{z_\lambda - z^\ast, z^k - z_\lambda} +2\left(1 - \frac{1}{2\lambda}\right) \|z_\lambda - z^k\|^2\right) \\
&\leq \frac{\|z_{1/2} - z^\ast\|\|z^k - z_{1/2}\|}{2\gamma} \leq \frac{\|z^0 - z^\ast\|\|z^k - z_{1/2}\|}{2\gamma} \stackrel{\eqref{cor:DRSaveragedconvergence:eq:main}}{\leq} \frac{\|z^0 - z^\ast\|^2}{4\gamma \sqrt{\underline{\tau}(k+1)}}.
\end{align*}
The little-$o$ convergence rate follows because $S_{f}(x_f^{k}, x^\ast) + S_g(x_g^{k}, x^\ast)$ is bounded by a multiple of the square root of the FPR.
\qed\end{proof}

It is not clear whether the ``best iterate" convergence results of Theorem~\ref{prop:sumauxilliaryterms} can be improved to a convergence rate for the entire sequence because the values $S_f(x_f^k, x)$ and $S_g(x_g^k, x)$ are not necessarily monotonic.

\section{Lipschitz derivatives}\label{sec:lipschitzderivatives}

In this section, we study the convergence rate of relaxed PRS under the following assumption.
\begin{assump}
The gradient of at least one of the functions $f$ and $g$ is Lipschitz.
\end{assump}

Throughout this section, Fact~\ref{lem:sumsequence} will be used repeatedly to deduce the convergence rates of summable sequences. In general, because we can only deduce the summability and not the monotonicity of the objective errors in Problem~\ref{eq:simplesplit}, we can only show that the smallest objective error after $k$ iterations is of order $o(1/(k+1))$.  If $\lambda_k \equiv 1/2$, the implicit stepsize parameter $\gamma$ is small enough, and the gradient of $g$ is $(1/\beta)$-Lipschitz, we show that a sequence that dominates the objective error is monotonic and summable, and deduce a convergence rate for the entire sequence.

\subsection{The general case: best iterate convergence rate}\label{sec:lipschitzderivatives:part:general}

%

The next proposition bounds the objective error by a summable sequence. See Appendix~\ref{app:lipschitzderivatives:part:general} for a proof.
\begin{proposition}[Fundamental inequality under Lipschitz assumptions]\label{prop:DRSupperfglipschitz}
Let $z \in \cH$,  let $z^+ = (\TPRS)_{\lambda}z$, let $z^\ast$ be a fixed point of $\TPRS$, and let $x^\ast = \prox_{\gamma g}(z^\ast)$.  If $\nabla f$ (respectively $\nabla g$) is $({1}/{\beta})$-Lipschitz, then for $x = x_g$ (respectively $x = x_f$),
\begin{align*}
&4\gamma\lambda\big(f(x) + g(x) - f(x^\ast) - g(x^\ast)\big)  \\
&\leq\begin{cases}
\|z - z^\ast\|^2 - \|z^+ - z^\ast\|^2 +\left(1 + \frac{1}{2\lambda}\left( \frac{\gamma}{\beta} - 1\right)\right)\|z - z^+\|^2, & \mbox{if } \gamma \leq \beta; \\
\left(1 + \frac{(\gamma - \beta)}{2\beta}\right)\left(\|z - z^\ast\|^2 - \|z^+ - z^\ast\|^2 + \|z - z^+\|^2\right), & \mbox{otherwise.}
\end{cases}
\end{align*}
\end{proposition}

Proposition~\ref{prop:DRSupperfglipschitz} shows that the the objective error is summable whenever $f$ or $g$ is Lipschitz and $(\lambda_j)_{j \geq 0}$ is chosen properly.  A direct application of Fact~\ref{lem:sumsequence} yields a convergence rate for the objective error.  Depending on the choice of $\gamma$ and $(\lambda_j)_{j \geq 0}$, we can achieve several different rates.  In the following Theorem we only analyze a few such choices.

\begin{theorem}[Best iterate convergence under Lipschitz assumptions]\label{thm:lipschitzbest}
Let $z \in \cH$,  
let $z^\ast$ be a fixed point of $\TPRS$, and let $x^\ast = \prox_{\gamma g}(z^\ast)$. Suppose that $\underline{\tau} = \inf_{j \geq 0} \lambda_j(1-\lambda_j) > 0$, and let $\underline{\lambda} = \inf_{j \geq 0}\lambda_j$.  If $\nabla f$ (respectively $\nabla g$) is $({1}/{\beta})$-Lipschitz, and $x^k = x_g^k$ (respectively $x^k = x_f^k$), then
\begin{align*}
\min_{i = 0, \cdots, k}\left\{f(x^{i}) + g(x^{i}) - f(x^\ast) - g(x^\ast)\right\} = o\left(\frac{1}{k+1}\right).
\end{align*}
\end{theorem}
\begin{proof}
Fact~\ref{fact:averagedconvergence} proves the following bound:
\begin{align*}
\inf_{j \geq 0} \frac{1-\lambda_j}{\lambda_j}\sum_{i=0}^\infty \|z^k - z^{k+1}\|^2 \leq  \sum_{i=0}^\infty \tau_i\|\TPRS z^i - z^i\|^2 \leq \|z^0 - z^\ast\|^2.
\end{align*}
Therefore, the proof follows from Part~\ref{lem:sumsequence:part:nonmono} of Lemma~\ref{lem:sumsequence} applied to the  summable upper bound in Proposition~\ref{prop:DRSupperfglipschitz}, which bounds the objective error. Note that under different choices of $(\lambda_j)_{j \geq 0}$ and $\gamma$, we get the bounds:
\begin{align*}
&f(x^{\kbest}) + g(x^{\kbest}) - f(x^\ast) - g(x^\ast) \\
&\leq \frac{\|z^0 - z^\ast\|^2}{4\gamma\underline{\lambda}(k+1)} \times \begin{cases}
1, & \mbox{if } \gamma \leq \beta \mbox{ and } (\lambda_j)_{j \geq 0} \subseteq \left[ \underline{\lambda}, \frac{1}{2}\left(1- \frac{\gamma}{\beta}\right)\right];\\
1 + 1/\left(\inf_{j \geq 0} \frac{1-\lambda_j}{\lambda_j}\right), & \mbox{if } \gamma \leq \beta; \\
\left(1 + \frac{(\gamma - \beta)}{2\beta}\right)\left(1 +  1/\left(\inf_{j \geq 0} \frac{1-\lambda_j}{\lambda_j}\right)\right), & \mbox{otherwise.}
\end{cases}
\end{align*}
\qed\end{proof}

This result should be compared with the known  convergence properties of the FBS algorithm, which has order $o(1/(k+1))$ for a bounded $\gamma$, but may even fail to converge if $\gamma$ is too large. See Section~\ref{section:FBSpractical} for more on the distinction between FBS and relaxed PRS.

\subsection{Constant relaxation and better rates}\label{sec:constantrelaxationandbetterates}

In this section, we study the convergence rate of DRS under the assumption
\begin{assump}
The function $g$ is differentiable on $\cH$, the gradient $\nabla g$ is $({1}/{\beta})$-Lipschitz, and the sequence of relaxation parameters $(\lambda_j)_{j \geq 0}$ is constant and equal to ${1}/{2}$.
\end{assump}
With these assumptions, we will show that for a special choice of $\theta^\ast$ (Lemma~\ref{lem:howtochoosea}) and for $\gamma$ small enough, the following sequence is monotonic and summable (Propositions~\ref{prop:lipschitzmono} and~\ref{prop:lipschitzsum}):
\begin{align}\label{eq:monoandsum}
\left(2\gamma\left(f(x_f^j) + g(x_f^j) - f(x) - g(x)\right) +\theta^\ast\gamma^2 \| \nabla g(x_g^{j+1}) - \nabla g(x_g^j)\|^2+ \frac{(1-\theta^\ast)\gamma^2}{\beta^2}\|x_g^{j+1} - x_g^{j}\|^2\right)_{j \geq 0}.
\end{align}
We then use Fact~\ref{lem:sumsequence} to deduce $f(x_f^j) + g(x_f^j) - f(x) - g(x) = o(1/(k+1))$.

There are several other simpler monotonic and summable sequences that dominate the objective error. For example, if we choose $\theta^\ast = 1$, we can drop the last term in Equation~\eqref{eq:monoandsum}, but we can no longer use this sequence to help deduce the convergence rate of the FPR in Theorem~\ref{thm:differentiableFPR}. Thus, we choose to analyze the slightly complicated sequence in Equation~\eqref{eq:monoandsum} in order to provide a unified analysis for all results in this section.

We are now ready to deduce the objective error convergence rate for the DRS algorithm when $\nabla g$ is Lipschitz.  Our bounds show that $$\mbox{DRS is at least as fast as FBS whenever $\gamma$ is small enough.}$$  Additionally, we show that the convergence rate of the best iterate has essentially the same constant for a large range of $\gamma$. When $\gamma$ is large, the best iterate still enjoys the convergence rate $o(1/(k+1))$, albeit with a larger constant (Theorem~\ref{thm:lipschitzbest}). The rates we derive are the best possible for this algorithm, as shown by \cite[Theorem 12]{davis2014convergence}.

Because each step of the relaxed PRS algorithm is generated by a proximal operator, it may seem strange that the choice of stepsize $\gamma$ affects the convergence rate of relaxed PRS.  This is certainly not the case for the proximal point algorithm, which achieves an $o(1/(k+1))$ convergence rate by Fact~\ref{lem:sumsequence}.  A possible explanation is that the reflection operator of a differentiable function is the composition of averaged operators
\begin{align*}
\refl_{\gamma g} = (I - \gamma \nabla g) \circ \prox_{\gamma g}
\end{align*}
whenever $\gamma < 2\beta$, and, therefore, it is averaged \cite[Propositions 4.32 and~4.33]{bauschke2011convex}. Thus, although $\TPRS$ is not necessarily averaged when $f$ or $g$ is differentiable, the individual reflection operators enjoy a stronger contraction property \cite[Proposition 4.25]{bauschke2011convex} as long as $\gamma$ is small enough.  As soon as $\gamma$ is too large, we seem to lose monotonicity of various sequences that arise in our analysis.

\begin{theorem}[Differentiable function convergence rate]\label{thm:differentiableobjective}
Let $\rho \approx 2.2056$ be the positive real root of $x^3 - 2x^2 - 1$. Then
\begin{align*}
\min_{i=0, \cdots, k} \{f(x_f^{i}) &+ g(x_f^{i}) - f(x^\ast) - g(x^\ast)\} \\
&\leq \frac{1}{2\gamma(k+1)}
\begin{cases}
\|x_g^0 - x^\ast\|^2, & \mbox{if $\gamma < \rho\beta$};\\
\|x_g^0 - x^\ast\|^2 + \frac{1}{\beta^2+ \gamma^2}\left( \frac{\gamma^3}{\beta} - 2\gamma \beta - \beta^2\right) \|z^0 - z^\ast\|^2, & \mbox{otherwise};
\end{cases}
\end{align*}
and $\min_{i=0, \cdots, k} \{f(x_f^{i}) + g(x_f^{i}) - f(x^\ast) - g(x^\ast)\} = o\left(1/(k+1)\right).$
Furthermore, if $\kappa$ ($\approx 1.24698$) is the positive root of $x^3 + x^2 -2x - 1$, and $\gamma < \kappa \beta$, then
\begin{align*}
f(x_f^k) + g(x_f^k) - f(x^\ast) - g(x^\ast) \leq \frac{\|x_g^0 - x^\ast\|^2}{2\gamma(k+1)} && \mathrm{and} &&
f(x_f^k) + g(x_f^k) - f(x^\ast) - g(x^\ast) = o\left(\frac{1}{k+1}\right).
\end{align*}
\end{theorem}
\begin{proof}
To prove the ``$\kbest$" bounds, rearrange the upper inequality in Equation~\eqref{prop:fundamentaldiff:eq:main} to
\begin{align*}
2\gamma&(f(x_f^k) + g(x_f^k) - f(x^\ast) - g(x^\ast))  \\
&\leq \|x_g^k - x^\ast\| - \|x_g^{k+1} - x^\ast\|^2 + \left(\frac{\gamma^3}{\beta} - 2\gamma \beta \right)\|\nabla g(x_g^{k+1}) - \nabla g(x_g^k)\|^2  - \|x_g^k - x_{g}^{k+1}\|^2 \\
&\leq \|x_g^k - x^\ast\| - \|x_g^{k+1} - x^\ast\|^2 + \left(\frac{\gamma^3}{\beta} - 2\gamma \beta - \beta^2\right)\|\nabla g(x_g^{k+1}) - \nabla g(x_g^k)\|^2, \numberthis \label{thm:differentiableobjective:bound1}
\end{align*}
where the last line follows from the bound $-\|x_g^k - x_{g}^{k+1}\|^2 \leq -\beta^2 \|\nabla g(x_g^{k}) - \nabla g(x_g^{k+1})\|^2.$ Note that ${\gamma^3}/{\beta} - 2\gamma \beta - \beta^2  \leq 0$ if, and only if,  $\gamma \leq \rho\beta$ where $\rho$ is the positive root of $x^3 - 2x^2 - 1$. Therefore, the result follows by summing Equation~\eqref{thm:differentiableobjective:bound1} and applying~Fact~\ref{lem:sumsequence}.

If $\gamma \leq \kappa\beta$, then $( {\gamma^3}/{\beta} - 2\gamma \beta + \theta^\ast\gamma^2) \leq 0$ and $(1-\theta^\ast)\gamma^2/\beta^2 \leq 1$. Therefore, Equation~\eqref{prop:lipschitzmono:eq:main} shows that the sequence
\begin{align*}
\left(2\gamma(f(x_f^j) + g(x_f^j) - f(x) - g(x)) +\theta^\ast\gamma^2 \| \nabla g(x_g^{j+1}) - \nabla g(x_g^j)\|^2+ \frac{(1-\theta^\ast)\gamma^2}{\beta^2}\|x_g^{j+1} - x_g^{j}\|^2\right)_{j \geq 0}
\end{align*}
is monotonic. In addition, Equation~\eqref{prop:lipschitzsum:eq:main} shows the sum of this sequence is bounded by $\|x_g^0 - x^\ast\|^2$.  Therefore, the result follows by~Fact~\ref{lem:sumsequence}.
\qed\end{proof}

It was recently shown that the FPR convergence rate for the FBS algorithm is $o(1/(k+1)^2))$ \cite[Theorem 3]{davis2014convergence}.  We complement this result by showing the same is true for DRS whenever $\gamma$ is small enough.  This rate is optimal by \cite[Theorem 12]{davis2014convergence}.

\begin{theorem}[Differentiable function FPR rate]\label{thm:differentiableFPR}
Suppose that $\gamma < \kappa \beta$ where $\kappa$ ($\approx 1.24698$) is the positive root of $x^3 + x^2 -2x - 1$. Then for all $k \geq 1$, we have
\begin{align}\label{eq:differentiableFPR}
\|z^k - z^{k+1}\|^2 \leq \frac{\beta^2\|x_g^0 - x^\ast\|^2}{k^2\left(1+{\gamma}/{\beta}\right)^2\left(\beta^2 -{\gamma^2}/{\kappa^2}\right)} && \mathrm{and} &&  \|z^k - z^{k+1}\|^2 &= o\left(\frac{1}{k^2}\right).
\end{align}
\end{theorem}
\begin{proof}
For all $k \geq 1$, let $$\eta = 1- \frac{(1-\theta^\ast)\gamma^2}{\beta^2} = \frac{\beta^2 +(\theta^\ast-1) \gamma^2}{\beta^2}= \frac{\beta^2 - \gamma^2/\kappa^2}{\beta^2}, $$ let $a_{k-1} = (\eta/(1+\gamma/\beta)^2)\|z^{k+1} - z^k\|^2$, and let
\begin{align*}
b_{k-1} =  2\gamma(f(x_f^k) + g(x_f^k) - f(x^\ast) - g(x^\ast)) + \theta^\ast\gamma^2 \| \nabla g(x_g^{k+1}) - \nabla g(x_g^k)\|^2 +\frac{(1-\theta^\ast)\gamma^2}{\beta^2}\|x_g^{k+1} - x_g^{k} \|^2.
\end{align*}
 Because $z^{k} = x_g^k + \gamma \nabla g(x_g^k)$ and and $\nabla g$ is $(1/\beta)$-Lipschitz, we get
\begin{align*}
\eta\|z^k - z^{k+1}\|^2 &\leq \eta\left(1+\frac{\gamma}{\beta}\right)^2\|x_g^k - x_g^{k+1}\|^2.
\end{align*}
Therefore, Equation~\eqref{prop:lipschitzmono:eq:main} shows that for all $k \geq 1$,
\begin{align*}
a_{k-1} \leq \eta \|x_g^k - x_g^{k+1}\|^2 \leq b_{k-1} - b_{k}.
\end{align*}
Fact~\ref{lem:sumsequence} applied to the sequences $(a_j)_{j \geq 0}$ and $(b_j)_{j \geq 0}$ with weighting parameters $\lambda_k \equiv 1$, (not to be confused with the constant relaxation parameter of the relaxed PRS algorithm), yields
\begin{align*}
\sum_{i=0}^\infty (i+1)a_i \leq \sum_{i=0}^\infty b_{i}  \stackrel{\eqref{prop:lipschitzsum:eq:main}}{\leq} \|x_g^0 - x^\ast\|^2.
\end{align*}
\cite[Part 2 of Theorem 1]{davis2014convergence} shows that $(a_j)_{j \geq 0}$ is monotonic. Therefore, the result follows from Fact~\ref{lem:sumsequence}.
\qed\end{proof}

\begin{remark}
Note that the FBS algorithm achieves $o(1/(k+1))$ objective error rate and $o(1/(k+1)^2)$ FPR rate as long as $\gamma < 2\beta$ \cite[Theorem 3]{davis2014convergence}. For the DRS algorithm, our analysis only covers the smaller range $\gamma \leq \kappa \beta$. It is an open question whether $\kappa$ can be improved for the DRS algorithm.
\end{remark}

\section{Linear convergence}\label{sec:linearconvergence}

In this section, we study the convergence rate of relaxed PRS under the assumption
\begin{assump}\label{assump:mixed}
The gradient of at least one of the functions $f$ and $g$ is Lipschitz, and at least one of the functions $f$ and $g$ is strongly convex.  In symbols: $(\mu_f + \mu_g)(\beta_f + \beta_g) > 0$.
\end{assump}
Linear convergence of relaxed PRS is expected whenever Assumption~\ref{assump:mixed} is true. In addition, by the strong convexity of $f + g$, the minimizer of Problem~\eqref{eq:simplesplit} is unique.

The following proposition lists some consequences of linear convergence of the relaxed PRS sequence $(z^j)_{j \geq 0}$.
\begin{proposition}[Consequences of linear convergence]\label{prop:linearconvergenceimplies}
Let $(C_j)_{j \geq0} \subseteq [0, 1]$ be a positive scalar sequence, and suppose that for all $k \geq0$,
\begin{align}\label{prop:linearconvergenceimplies:eq:main}
\|z^{k+1} - z^\ast\| \leq C_k\|z^k - z^\ast\|.
\end{align}
Fix $k \geq 1$. Then
\begin{align*}
 \|x_g^k - x^\ast\|^2 + \gamma^2\|\tnabla g(x_g^k) - \tnabla g(x^\ast)\|^2 \leq \|z^0 - z^\ast\|^2\prod_{i = 0}^{k-1} C_i^2; \\
 \|x_f^k - x^\ast\|^2 + \gamma^2\|\tnabla f(x_f^k) - \tnabla f(x^\ast)\|^2 \leq \|z^0 - z^\ast\|^2\prod_{i = 0}^{k-1} C_i^2.
 \end{align*}
If $\lambda <1$, then the FPR rate holds: $\|(\TPRS)_{\lambda}z^k- z^k\| \leq \sqrt{\lambda/(1-\lambda)}\|z^0 - z^\ast\|\prod_{i = 0}^{k-1} C_i.$
Consequently, if the gradient $\nabla f$ (respectively $\nabla g$), is $({1}/{\beta})$-Lipschitz and $x^k = x_g^k$ (respectively $x^k = x_f^k$), then
\begin{align*}
&f(x^k) + g(x^k) - f(x^\ast) - g(x^\ast) \leq \frac{\|z^0 - z^\ast\|^2}{\gamma} \prod_{i=0}^{k-1}C_i^2 \times\begin{cases}
1  , & \mbox{if } \gamma \leq \beta; \\
1 + \frac{(\gamma - \beta)}{2\beta}, & \mbox{otherwise.}\\
\end{cases}
\end{align*}
\end{proposition}
\begin{proof}
The bounds for $x_g^k$ and $x_f^k$ follow because
$\|x_g^k - x^\ast\|^2 + \gamma^2\|\nabla g(x_g^k)  -\nabla g(x^\ast)\|^2 \leq \|z^k - z^\ast\|^2,$
 and  $\|x_f^k - x^\ast\|^2 + \gamma^2\|\tnabla f(x_f^k) - \tnabla f(x^\ast)\|^2 \leq \|\refl_{\gamma g}(z^k) - \refl_{\gamma g}(z^\ast)\|^2 \leq \|z^k - z^\ast\|^2$
by Part~\ref{cor:proxcontraction} of Proposition~\ref{prop:basicprox}, the nonexpansiveness of $\refl_{\gamma f}$, and Equation~\eqref{prop:linearconvergenceimplies:eq:main}.

The FPR convergence rate follows from the Fej\'er-type inequality in Equation~\eqref{eq:fejer}.

Now fix $k \geq 1$, and let $z_\lambda = (\TPRS)_\lambda z^k$ for all $\lambda \in [0, 1]$. Then Proposition~\ref{prop:DRSupperfglipschitz} shows that:
\begin{align*}
&f(x^k) + g(x^k) - f(x^\ast) - g(x^\ast) \\
&\leq \inf_{\lambda \in [0, 1]}\frac{1}{4\gamma \lambda}\begin{cases}
\|z^k - z^\ast\|^2 - \|z_\lambda - z^\ast\|^2 +\left(1 + \frac{1}{2\lambda}\left( \frac{\gamma}{\beta} - 1\right)\right)\|z^k - z_\lambda\|^2, & \mbox{if } \gamma \leq \beta; \\
\left(1 + \frac{(\gamma - \beta)}{2\beta}\right)\left(\|z^k - z^\ast\|^2 - \|z_\lambda - z^\ast\|^2 + \|z^k - z_\lambda\|^2\right), & \mbox{otherwise.}
\end{cases} \\
&\leq \frac{1}{2\gamma}\begin{cases}
\|z^k - z^\ast\|^2+ \|z^k - z_{1/2}\|^2  , & \mbox{if } \gamma \leq \beta; \\
\left(1 + \frac{(\gamma - \beta)}{2\beta}\right)\left(\|z^k - z^\ast\|^2  + \|z^k - z_{1/2}\|^2\right), & \mbox{otherwise.}\\
\end{cases} \numberthis\label{eq:preboundlinear}
\end{align*}
The objective error rate now follows from Equation~\eqref{eq:preboundlinear} and the FPR convergence rate.
\qed\end{proof}

Whenever $\sup_{j \geq 0} C_j < 1$, Proposition~\ref{prop:linearconvergenceimplies} gives the linear convergence rates of the sequences $(z^j)_{j \geq 0}$, $(x_g^j)_{j \geq 0}$ and $(x_f^j)_{j \geq 0}$, the subgradient error, the FPR, and the objective error. In the following sections, we will prove Inequality~\eqref{prop:linearconvergenceimplies:eq:main} holds under several different regularity assumptions on $f$ and $g$. In each case we leave it to the reader to apply Proposition~\ref{prop:linearconvergenceimplies}.

\subsection{Solely regular $f$ or $g$}\label{sec:strongconvexity}

Throughout this subsection, at least one of the functions $f$ and $g$ will carry both regularity properties. In symbols: $\mu_f\beta_f + \mu_g\beta_g > 0$.

The following theorem recovers \cite[Proposition 4]{lions1979splitting} as a special case ($\lambda_k \equiv 1/2$).

\begin{theorem}[Linear convergence with regularity of $g$]\label{thm:nonergodicstronglipschitzderivative}
Let $z^\ast$ be a fixed point of $\TPRS$, let $x^\ast = \prox_{\gamma g}(z^\ast)$, and suppose that $\mu_g\beta_g > 0$. For all $\lambda \in [0, 1]$, let $C(\lambda) := \left(1 - {4\gamma\lambda\mu_g}/{(1 + {\gamma}/{\beta_g})^2}\right)^{{1}/{2}}$.
Then for all $k \geq 0$,
$\|z^{k+1} - z^\ast\| \leq  C(\lambda_k)   \|z^k - z^\ast\|.$
\end{theorem}
\begin{proof}
Theorem~\ref{prop:sumauxilliaryterms} bounds the distance of $x_g^k$ to  the minimizer
\begin{align*}
\frac{8\gamma\lambda_k\mu_g}{2}\|x_g^k - x^\ast\|^2 &\stackrel{\eqref{prop:sumauxilliaryterms:eq:main}}{\leq}  \|z^{k} - z^\ast\|^2 - \|z^{k+1} - z^\ast\|^2. \numberthis \label{thm:nonergodicstronglipschitzderivative:eq:2}
\end{align*}
Now we use the identity $z^{k} = x_g^k + \gamma \nabla g(x_g^k)$ and the Lipschitz continuity of $\nabla g$ to upper bound $\|z^k - z^\ast\|^2$ by a multiple of $\|x_g^k - x^\ast\|^2$:
$\|z^k - z^\ast\|^2 \leq \left(1+\gamma/\beta_g\right)^2\|x_g^k - x^\ast\|^2.$
Rearrange Equation~\eqref{thm:nonergodicstronglipschitzderivative:eq:2} with this bound to complete the proof.
\qed\end{proof}

\begin{remark}
For all $\lambda \in [0, 1]$, the constant $C(\lambda)$ is minimal when $\gamma = \beta_g$, i.e. $C(\lambda) = \left( 1- \lambda_k\mu_g\beta_g\right)^{{1}/{2}}$. Furthermore, for any choice of $\gamma$, we have the bound $C(1) \leq C(\lambda)$. In particular, for $g = ({1}/{2})\|\cdot \|^2$, the PRS algorithm converges in one step ($C(1) = 0$). Thus, this rate is tight.
\end{remark}

The following theorem deduces linear convergence of relaxed PRS whenever $f$ carries both regularity properties. Note that linear convergence of the PRS algorithm ($\lambda_k \equiv 1$) does not follow.

\begin{theorem}[Linear convergence with regularity of $f$]\label{prop:regf}
Let $z^\ast$ be a fixed point of $\TPRS$, let $x^\ast = \prox_{\gamma g}(z^\ast)$, and suppose that $\mu_f\beta_f > 0$. For all $\lambda \in [0, 1]$, let $$C(\lambda) := \left(1 -   (\lambda/2)\min\left\{ {4\gamma  \mu_f}/{\left(1+{\gamma}/{\beta_f}\right)^2}, (1 - \lambda)\right\}\right)^{{1}/{2}}.$$ Then for all $k \geq 0$,
$\|z^{k+1} - z^\ast\| \leq  C(\lambda_k)   \|z^k - z^\ast\|.$
\end{theorem}
\begin{proof}
Theorem~\ref{prop:sumauxilliaryterms} bounds the distance of $x_f^k$ to  the minimizer (where we substitute $z^{k+1} - z^k = 2\lambda_k(x_f^k - x_g^k)$)
\begin{align*}
4\gamma\lambda_k\mu_f\|x_f^k - x^\ast\|^2 + 4\lambda_k\left(1-\lambda_k \right) \|x_f^k - x_g^k\|^2&\stackrel{\eqref{prop:sumauxilliaryterms:eq:main}}{\leq} \|z^k - z^\ast\|^2 - \|z^{k+1} - z^\ast\|^2. \numberthis \label{prop:regf:eq:first}
 \end{align*}
Recall the identities:
\begin{align*}
z^k = x_g^k + \gamma \nabla g(x_g^k) = x_f^k - \gamma \nabla f(x_f^k) + 2(x_g^k - x_f^k) && \mathrm {and} && z^\ast = x^\ast - \gamma \nabla f(x^\ast).
\end{align*}
Therefore, by the convexity of $\|\cdot \|^2$, we can bound the distance of $z^k$ to the fixed point $z^\ast$
\begin{align*}
\|z^k - z^\ast\|^2 \leq 2\left(\left(1+\frac{\gamma}{\beta_f}\right)^2\|x_f^k - x^\ast\|^2  + 4\|x_g^k - x_f^k\|^2\right). \numberthis\label{prop:regf:eq:second}
\end{align*}
Equations~\eqref{prop:regf:eq:first} and~\eqref{prop:regf:eq:second} produce the contraction:
\begin{align*}
C'\|z^k - z^\ast \|^2 + \|z^{k+1} - z^\ast\|^2 \leq 4\gamma\lambda_k\mu_f\|x_f^k - x^\ast\|^2 + 4\lambda_k(1 - \lambda_k)\|x_f^k - x_g^k\|^2  + \|z^{k+1} - z^\ast\|^2 \leq \|z^k - z^\ast\|^2
\end{align*}
where $C' = {(\lambda_k/2)\min\left\{ {4\gamma  \mu_f}/{\left(1+{\gamma}/{\beta_f}\right)^2}, (1 - \lambda_k)\right\}}.$ 
\qed\end{proof}

\subsection{Complementary regularity of $f$ and $g$}

In this subsection, we assume that $f$ and $g$ share the regularity. In symbols: $\mu_f\beta_g + \mu_g\beta_f > 0$. In this case, linear convergence is expected. To the best of our knowledge, the next result is new.

\begin{theorem}[Linear convergence: mixed case]\label{thm:sharedreg}
Let $z^\ast$ be a fixed point of $\TPRS$, let $x^\ast = \prox_{\gamma g}(z^\ast)$, and suppose that $\nabla g$, (respectively $\nabla f$), is $({1}/{\beta})$-Lipschitz and $f$, (respectively $g$), is $\mu$-strongly convex.  For all $\lambda \in [0, 1]$, let $C(\lambda) :=  \left(1 -   {(4\lambda/3)\min\{ \gamma  \mu, { \beta}/{\gamma}, (1 - \lambda)\}}\right)^{{1}/{2}}$. Then for all $k \geq 0$,
$\|z^{k+1} - z^\ast\| \leq  C(\lambda_k)\|z^k - z^\ast\|.$
\end{theorem}
\begin{proof}
First assume that $\mu_f\beta_g > 0$. Theorem~\ref{prop:sumauxilliaryterms} bounds the distance of $x_f^k$ to the minimizer and the distance of $\nabla g(x_g^k)$ to the optimal gradient (where we substitute $z^{k+1} - z^k = 2\lambda_k(x_f^k - x_g^k)$):
\begin{align*}
4\gamma\lambda_k\mu\|x_f^k - x^\ast\|^2 &+ 4\gamma\lambda_k \beta\|\nabla g(x_g^k) - \nabla g(x^\ast)\|^2 + 4\lambda_k\left(1- \lambda_k\right) \|x_f^k - x_g^k\|^2\\
 &\stackrel{\eqref{prop:sumauxilliaryterms:eq:main}}{\leq} \|z^k - z^\ast\|^2 - \|z^{k+1} - z^\ast\|^2. \numberthis\label{thm:sharedreg:eq:first}
\end{align*}
Recall the identities:
\begin{align*}
z^k = x_g^k + \gamma \nabla g(x_g^k) = x_f^k + \gamma \nabla g(x_g^k) + (x_g^k - x_f^k) && \mathrm{and} && z^\ast = x^\ast + \gamma \nabla g(x^\ast).
\end{align*}
Thus, from the convexity of $\|\cdot \|^2$,
\begin{align*}
\|z^k - z^\ast\|^2 \leq 3\left(\|x_f^k - x^\ast\|^2 + \|\gamma \nabla g(x_g^k) - \gamma \nabla g(x^\ast)\|^2 + \|x_g^k - x_f^k\|^2\right) \numberthis \label{thm:sharedreg:eq:second} .
\end{align*}
We use Equation~\eqref{thm:sharedreg:eq:second} to bound the distance of $z^k$ to the fixed point $z^\ast$ by the left hand side of Equation~\eqref{thm:sharedreg:eq:first}:
\begin{align*}
C'\|z^k - z^\ast \|^2 &\leq 4\gamma\lambda_k\mu\|x_f^k - x^\ast\|^2 + 4\lambda_k ({ \beta}/{\gamma})\|\gamma\nabla g(x_g^k) - \gamma\nabla g(x^\ast)\|^2 + 4\lambda_k(1 - \lambda_k)\|x_f^k - x_g^k\|^2
\end{align*}
where $C' = {(4\lambda_k/3)\min\{ \gamma  \mu, { \beta}/{\gamma}, (1 - \lambda_k)\}}.$  Therefore, we reach the contraction:
\begin{align*}
\|z^{k+1} - z^\ast\| \leq  \left(1 -   {(4\lambda_k/3)\min\{ \gamma  \mu, { \beta}/{\gamma}, (1 - \lambda_k)\}}\right)^{{1}/{2}}\|z^k - z^\ast\|^2.
\end{align*}

If $\mu_g\beta_f> 0$, then the proof is nearly identical, but relies on the identity:
\begin{align*}
z^k = x_g^k + \gamma \tnabla g(x_g^k) = x_g^k - \gamma \nabla f(x_f^k) + (x_g^k - x_f^k).
\end{align*}
\qed\end{proof}

\section{Feasibility Problems with regularity}\label{sec:feasibilityniceintersection}
In this section we consider the feasibility problem: $$\mbox{Given two closed convex subsets $C_f$ and $C_g$ of $\cH$ such that $C_f \cap C_g \neq \emptyset$, find a point $x \in C_f \cap C_g$.}$$  Throughout this section we assume that $\{C_f, C_g\}$ is \emph{boundedly linearly regular}:
\begin{definition}[Bounded linear regularity]\label{defi:linearregularity}
Suppose that $C_1, \cdots, C_m$ are closed convex subsets of $\cH$ with nonempty intersection.  We say that $\{C_1, \cdots, C_m\}$ is boundedly linearly regular if the following holds: for all $\rho > 0$, there exists $\mu_\rho > 0$ such that for all $x \in B(0, \rho)$, (the open ball centered at the origin with radius $\rho$), we have
\begin{align*}
d_{C_1\cap \cdots \cap C_m}(x) &\leq \mu_\rho\max\{d_{C_1}(x), \cdots, d_{C_m}(x)\}
\end{align*}
where for any subset $C \subseteq \cH$, the distance function $d_C(x):=  \inf_{y \in C} \|x -y\|.$
Evidently, if $B(0, \rho) \backslash (C_1 \cap \cdots \cap C_m) \neq \emptyset$, then $\mu_\rho \geq 1$.

We say that $\{C_1,\cdots, C_m\}$ is linearly regular if it is boundedly linearly regular and $\mu_\rho$ does not depend on $\rho$, i.e. $\mu_\rho = \mu_\infty < \infty$.
\qed\end{definition}

Intuitively, (bounded) linear regularity is the following implication: $$\mbox{(close to all of the sets) $\implies$ (close to the intersection).}$$ This property will be key to deducing linear convergence of an application of the relaxed PRS algorithm.  See \cite{davis2014convergence} for the feasibility problem when no regularity is assumed.

There are several ways to model the feasibility problem, e.g. with $f$ and $g$ given by indicator functions, distance functions, or squared distance functions. In this section, we will model the feasibility problem using squared distance functions:
\begin{align*}
f(x) := d_{C_f}^2(x)  && \mathrm{and} && g(x) := d_{C_g}^2(x).
\end{align*}
We briefly summarize some properties of squared distance functions.

\begin{proposition}[Properties of distance functions]\label{prop:factsaboutdistancesquared}
Let $C$ be a nonempty closed convex subset of $\cH$. Then the following properties hold:
\begin{enumerate}
\item The function $d_{C}$ is $1$-Lipschitz.
\item The function $d_{C}^2$ is differentiable, and $\nabla d_C^2 = 2(I_{\cH} - P_{C})$. In addition, $\nabla d_C^2$ is $2$-Lipschitz.
\item The proximal identity holds: for all $\gamma > 0$,
\begin{align*}
\prox_{ \gamma d^2_C} &= \frac{1}{2\gamma + 1}I_{\cH} + \frac{2\gamma}{2\gamma + 1} P_{C}.
\end{align*}
\end{enumerate}
\end{proposition}
\begin{proof}
For a proof see \cite[Corollary 12.30]{bauschke2011convex}.
\qed\end{proof}

Given $z^0 \in \cH$, sequences of implicit stepsize parameters, $(\gamma_{f, j})_{j \geq 0}$, $(\gamma_{g, j})_{j \geq 0}$, and relaxation parameters, $(\lambda_j)_{j \geq 0}$, we consider the iteration: for all $k \geq 0$, let
\begin{align}\label{eq:DRSfeasibilitygamma}
\begin{cases}
x_g^k &= \prox_{\gamma_{g, k} d_{C_g}^2} (z^k); \\ 
x_f^k &= \prox_{\gamma_{f,k} d_{C_f}^2}(2x_g^k - z^k); \\ 
z^{k+1} &= z^k + 2\lambda_k(x_f^k - x_g^k).
\end{cases}
\end{align}
If $(\gamma_{f,j})_{j \geq 0}, (\gamma_{g, j})_{j \geq 0} \subseteq (0, {1}/{2}]$ and $\lambda_k \equiv 1$, then the iteration in Equation~\eqref{eq:DRSfeasibilitygamma} is the \emph{underrelaxed MAP} (see \cite{bauschke1996projection} for the parallel product space version and see \cite{bauschke2013method} for the nonconvex case).  In particular, Corollary~\ref{cor:AP} (below) shows that when all implicit stepsize parameters are equal to ${1}/{2}$ and all relaxation parameters are $1$, Equation~\eqref{eq:DRSfeasibilitygamma} reduces to the MAP algorithm, where $ P_{C_g}z^k = 2x_g^k - z^k$, and $z^{k+1} = P_{C_f}P_{C_g} z^k$.  This was already noticed in \cite[Proposition 2.5]{luke2008finding} for the fixed $\gamma$ case.

We now specialize the fundamental inequality in Proposition~\ref{prop:DRSupper} to the feasibility problem. See Appendix~\ref{app:feasibilityniceintersection} for a proof.
\begin{proposition}[Upper fundamental inequality for feasibility problem]\label{eq:DRSupperfeas}
Suppose that $z\in \cH$ and $z^+ = (\TPRS^{\gamma_f, \gamma_g})_{\lambda}(z)$. Then for all $x^\ast\in C_f\cap C_g$,
\begin{align}\label{eq:DRSupperfeas:eq:main}
8\lambda(\gamma_{f} d^2_{C_f}(x_f) + \gamma_{g}d^2_{C_g}(x_g))  &\leq \|z -  x^\ast\|^2 -  \|z^{+} - x^\ast\|^2 + \left(1 - \frac{1}{\lambda} \right)\|z^{+} - z\|^2.
\end{align}
\end{proposition}

We are now ready to prove the linear convergence of Algorithm~\eqref{eq:DRSfeasibilitygamma} whenever $\{C_f, C_g\}$ is (boundedly) linearly regular.  The proof is a consequence of the upper inequality in Proposition~\ref{eq:DRSupperfeas}.

\begin{theorem}[Linear convergence: Feasibility for two sets]\label{thm:linearfeasibility}
Suppose that $(z^j)_{j \geq 0}$ is generated by the iteration in Equation~\eqref{eq:DRSfeasibilitygamma}, and that $C_f$ and $C_g$ are (boundedly) linearly regular. Let $\rho > 0$ and $\mu_\rho > 0$ be such that $(z^j)_{j \geq 0} \subseteq B(0, \rho)$ and the inequality
\begin{align*}
d_{C_f \cap C_g}(x) &\leq \mu_\rho\max\{ d_{C_f}(x), d_{C_g}(x)\}
\end{align*}
holds  for all $x \in B(0, \rho)$. Then $(z^j)_{j \geq 0}$ satisfies the following relation: for all $k \geq 0$,
\begin{align}\label{thm:linearfeasibility:eq:decrease}
d_{C_f \cap C_g}(z^{k+1}) &\leq C(\gamma_{f, k}, \gamma_{g, k}, \lambda_k, \mu_\rho) \times d_{C_f \cap C_g}(z^k)
\end{align}
where
\begin{align*}
C(\gamma_{f, k}, \gamma_{g, k}, \lambda_k, \mu_\rho) &:= \left(1-\frac{4\lambda_k\min\{{\gamma_{g, k}}/{(2\gamma_{g, k} + 1)^2}, {\gamma_{f, k}}/{(2\gamma_{f, k} + 1)^2}\}}{\mu_\rho^2\max\{{16\gamma_{g, k}^2}/{(2\gamma_{g, k}+1)^2}, 1\}} \right)^{{1}/{2}}.
\end{align*}

In particular, if $\overline{C} = \sup_{j \geq 0}C(\gamma_{f, j}, \gamma_{g, j}, \lambda_j, \mu) < 1$, then $(z^j)_{j \geq 0}$ converges linearly to a point in $x \in C_f \cap C_g$ with rate $\overline{C}$, and
\begin{align}
\|z^k - x\| &\leq 2d_{C_f\cap C_g}(z^0) \prod_{i=0}^k C(\gamma_{f, i}, \gamma_{g, i}, \lambda_i, \mu).
\end{align}
\end{theorem}
\begin{proof}
For simplicity, throughout the proof we will drop the iteration index $k$ and denote $z^+ := z^{k+1}$ and $z := z^k$, etc.  Now recall the identities:
\begin{align}
x_g = \frac{1}{2\gamma_{g} + 1} z + \frac{2\gamma_{g}}{2\gamma_{g}+1}P_{C_g}(z) &&\mathrm{and} && x_f = \frac{1}{2\gamma_{f} + 1} \refl_{\gamma_{g} g}(z) + \frac{2\gamma_{f}}{2\gamma_{f} + 1} P_{C_f} (\refl_{\gamma_{g} g}(z)).
\end{align}
Thus, $x_g$ is a point on the line segment connecting $P_{C_g}(z)$ and $z$, and $x_f$ is a point on the line segment connecting $\refl_{\gamma_g g}(z)$ and  $P_{C_f}(\refl_{\gamma_g g}(z))$. Hence,  we have the projection identities: $P_{C_g}z = P_{C_g}x_g$ and $P_{C_f}(\refl_{\gamma_{g} g}(z)) = P_{C_f} x_f$.  We can also compute the distances to $C_f$ and $C_g$:
\begin{align}\label{thm:linearfeasibility:eq:switchtoz}
d_{C_g}^2(x_g) = \frac{1}{(2\gamma_{g}+1)^2}d_{C_g}^2(z) && \mathrm{and} && d_{C_f}^2(x_f) = \frac{1}{(2\gamma_{f} + 1)^2}d^2_{C_f}(\refl_{\gamma_{g} g}(z)).
\end{align}

We will now bound $d^2_{C_f}(z)$.  Because $x_g$ is a point on the line segment connecting $z$ and $P_{C_g}(z)$, Equation~\eqref{thm:linearfeasibility:eq:switchtoz} shows that that $\|z - x_g\| = (2\gamma_g/(2\gamma_g + 1) ) d_{C_g}(z)$.  Thus, if $c_1 := c_1(\gamma_{g}) = {4\gamma_{g}}/({2\gamma_{g}+1})$, we have
\begin{align}\label{thm:linearfeasibility:eq:reflbound}
\|z - \refl_{\gamma_{g} g}(z)\| = 2\|z - x_g\| = c_1d_{C_g}(z).
\end{align}
Therefore, because $d_{C_f}$ is $1$-Lipschitz and by the convexity of $(\cdot)^2$,
\begin{align*}
d^2_{C_f}(z) &\leq (\|z - \refl_{\gamma_{g} g}(z)\| + d_{C_f}(\refl_{\gamma_{g} g}(z)))^2 \\
&= (c_1d_{C_g}(z) + d_{C_f}(\refl_{\gamma_{g} g}(z)))^2 \\
&\leq 2\max\{c_1^2, 1\}(d^2_{C_g}(z) + d^2_{C_f}(\refl_{\gamma_{g} g}(z))). \numberthis \label{thm:linearfeasibility:eq:upperboundztoc1}
\end{align*}

Now we will simplify the upper bound in Equation~\eqref{eq:DRSupperfeas:eq:main} by using Equation~\eqref{thm:linearfeasibility:eq:switchtoz}
\begin{align*}
8\lambda\left(\frac{\gamma_{g}}{(2\gamma_{g} + 1)^2}d^2_{C_f}(\refl_{\gamma_{g} g}(z)) + \frac{\gamma_{f}}{(2\gamma_{f} + 1)^2}d^2_{C_g}(z)\right) &+ \|z^{+} - x\|^2 + \left(\frac{1}{\lambda} - 1\right)\|z^{+} - z\|^2 \leq \|z - x\|^2. \numberthis \label{thm:linearfeasibility:eq:upper2}
\end{align*}
Because ${1}/({2\max\{c_1^2, 1\}}) < 1$, we have
\begin{align*}
&8\lambda\left(\frac{\gamma_{g}}{(2\gamma_{g} + 1)^2}d^2_{C_f}(\refl_{\gamma_{g} g}(z)) + \frac{\gamma_{f}}{(2\gamma_{f} + 1)^2}d^2_{C_g}(z)\right)\\
&\geq 8\lambda\min\left\{\frac{\gamma_{g}}{(2\gamma_{g} + 1)^2}, \frac{\gamma_{f}}{(2\gamma_{f} + 1)^2}\right\}\left(d^2_{C_f}(\refl_{\gamma_{g} g}(z)) + d^2_{C_g}(z)\right) \\
&\stackrel{\eqref{thm:linearfeasibility:eq:upperboundztoc1}}{\geq}\frac{8\lambda\min\{{\gamma_{g}}/{(2\gamma_{g} + 1)^2}, {\gamma_{f}}/{(2\gamma_{f} + 1)^2}\}}{2\max\{c_1^2, 1\}} \max\{d_{C_f}^2(z), d^2_{C_g}(z)\}.  \numberthis \label{thm:linearfeasibility:eq:lower}
\end{align*}

Now, recall the bounded linear regularity property: for all $x \in B(0, \rho)$,
\begin{align*}
d_{C_f \cap C_g}(x) &\leq \mu_\rho\max\{ d_{C_f}(x), d_{C_g}(x)\}.
\end{align*}
Thus, for all $x \in C_f \cap C_g$, the lower bound in Equation~\eqref{thm:linearfeasibility:eq:lower} shows that (where we use $(1/\lambda - 1) \geq 0$ in Equation~\eqref{thm:linearfeasibility:eq:upper2})
\begin{align*}
\frac{4\lambda\min\{{\gamma_{g}}/{(2\gamma_{g} + 1)^2}, {\gamma_{f}}/{(2\gamma_{f} + 1)^2}\}}{\mu_\rho^2\max\{c_1^2, 1\}} d^2_{C_f\cap C_g}(z) + \|z^{+} - x\|^2&\stackrel{\eqref{thm:linearfeasibility:eq:upper2}}{\leq} \|z - x\|^2.
\end{align*}
Hence, if we define $$C(\gamma_{f}, \gamma_{g}, \lambda, \mu_\rho) = \left(1-\frac{4\lambda\min\{{\gamma_{g}}/{(2\gamma_{g} + 1)^2}, {\gamma_{f}}/{(2\gamma_{f} + 1)^2}\}}{\mu_\rho^2\max\{c_1^2, 1\}} \right)^{{1}/{2}}$$ and $ x= P_{C_f \cap C_g}(z)$, then $d_{C_f \cap C_g}(z) = \|z - x\|$ and $d_{C_f \cap C_g}(z^{+}) \leq \|z^{+} - x\|$. Therefore,
\begin{align}
d_{C_f\cap C_g}(z^{+}) &\leq C(\gamma_{f}, \gamma_{g}, \lambda, \mu_\rho)d_{C_f\cap C_g}(z).
\end{align}

Linear convergence of $(z^j)_{j \geq 0}$ to a point in $C_f \cap C_g$ follows from \cite[Theorem 5.12]{bauschke2011convex}. The rate follows from Equation~\eqref{thm:linearfeasibility:eq:decrease}.
\qed\end{proof}

\begin{remark}
The recent papers \cite{bauschke2014linear,phan2014linear} have proved linear convergence of DRS applied to $f = \iota_{C_f}$ and $g = \iota_{C_g}$ under the same bounded linear regularity assumption on the pair $\{C_f, C_g\}$. In \cite{bauschke2014linear}, the proof uses the FPR to bound the distance of $z^k$ to the fixed point set of $\TPRS$.  Note that for any closed convex set $C$, we have the limit: $\prox_{\gamma d_{C}^2}(x) \rightarrow P_{C}(x)$ as $\gamma \rightarrow \infty$. Thus, the results of \cite{bauschke2014linear} and \cite{phan2014linear} can be seen as the limiting case of our results, but cannot be recovered from Theorem~\ref{thm:linearfeasibility}. Indeed, for any positive $\lambda$ and $\mu$, we have the limit: $C(\gamma', \gamma, \lambda, \mu) {\rightarrow} 1$, as $\gamma, \gamma' \rightarrow \infty.$
\end{remark}
\begin{remark}
The constant $C(\gamma, \gamma', \lambda, \mu)$ has the following form:
\begin{align*}
C(\gamma', \gamma, \lambda, \mu) &=
\begin{cases}
 \left(1 - \frac{\lambda(2\gamma + 1)^2\min\left\{{\gamma}/{(2\gamma + 1)^2}, {\gamma'}/{(2\gamma'  +1 )^2}\right\}}{4\gamma^2\mu^2}\right)^{{1}/{2}}, & \mbox{if } \gamma \geq \frac{1}{2}; \\
\left(1 - \frac{4 \lambda\min\left\{{\gamma}/{(2\gamma + 1)^2}, {\gamma'}/{(2\gamma'  +1 )^2}\right\}}{\mu^2}\right)^{{1}/{2}}, & \mbox{otherwise.}
\end{cases}
\end{align*}
For fixed positive $\gamma, \lambda$ and $\mu$, the function $C(\gamma', \gamma, \lambda, \mu)$ is minimized when $\gamma' = {1}/{2}$. Furthermore, it follows that that $C({1}/{2}, \gamma, \lambda, \mu)$ is minimized over $\gamma$, at $\gamma = {1}/{2}$. Finally, note that $C(\gamma', \gamma, \lambda, \mu)$ is monotonically decreasing in $\lambda$ and monotonically increasing in $\mu$.  Thus, in view of Corollary~\ref{cor:AP}, we achieve the minimal constant for MAP: $C({1}/{2}, {1}/{2}, 1, \mu) = \left(1 - {1}/({2\mu^2)}\right)^{{1}/{2}}$.
\end{remark}

We can use Theorem~\ref{thm:linearfeasibility} to deduce the linear convergence of MAP and give an explicit rate.  In \cite[Theorem 3.15]{deutsch2008rate}, the authors show that $\mu$-linear regularity of a finite collection of sets is equivalent to the linear convergence of the method of cyclic projections applied to these sets and, they derive the rate $\left(1 - {1}/({8\mu^2})\right)^{{1}/{2}}$.  Corollary~\ref{cor:AP} is a special case of one direction of this result but with a better rate. It is not clear if the rate in \cite[Theorem 3.15]{deutsch2008rate} can be improved for the general cyclic projections algorithm.  The rate we show in Corollary~\ref{cor:AP} appears in \cite[Corollary 3.14]{bauschke1993convergence} under the same assumptions.

\begin{corollary}[Convergence of MAP] \label{cor:AP}
Let $(z^j)_{j \geq 0}$ be generated by the iteration in Equation~\eqref{eq:DRSfeasibilitygamma} with  $\gamma_{f, k} \equiv \gamma_{g, k} \equiv {1}/{2}$ and $\lambda_k \equiv1$. Then for all $k \geq 0$, $z^{k+1} = P_{C_f}P_{C_g}z^k$. Thus, MAP is a special case of PRS.   Consequently, under the assumptions of Theorem~\ref{thm:linearfeasibility}, the iterates of MAP converge linearly to a point in the intersection of $C_f \cap C_g$ with rate $\left(1 - {1}/{\mu_\rho^2}\right)^{{1}/{2}}$.
\end{corollary}
\begin{proof}
Notice that $x_g^k = ({1}/{2})z^k + ({1}/{2})P_{C_g}z^k$ and $\refl_{({1}/{2})g}(z^k) = P_{C_g}z^k$.  Similarly, $x_f^k = ({1}/{2})P_{C_g}(z^k) + ({1}/{2})P_{C_f}P_{C_g}z^k$ and $z^{k+1} = \refl_{({1}/{2}) f}(P_{C_g}z^k) = P_{C_f}P_{C_g}z^k$.

We see that $C({1}/{2}, {1}/{2}, 1, \mu) = \left(1 - {1}/({2\mu_\rho^2})\right)^{{1}/{2}}$. We can strengthen this rate to $\left(1 - {1}/{\mu_\rho^2}\right)^{{1}/{2}}$ by observing that in Equation~\eqref{thm:linearfeasibility:eq:upperboundztoc1} we have $d_{C_f}(z^k) = 0$, and so we can set $c_1 = 0$. The proof then follows the same argument.
\qed\end{proof}

\begin{remark}
If $C_f$ and $C_g$ are closed subspaces with Friedrichs angle $\cos^{-1}(c_F)$, \cite[Corollary 11]{bauschke1999strong} shows that $\mu \leq {2}/{\sqrt{1 - c_F}}$. Therefore, Corollary~\ref{cor:AP} predicts that iterates of MAP converges with rate no less than $(({3 + c_F})/{4})^{1/2}$. The actual rate for this problem is $c_F^2$ \cite{aronszajn1950theory,kayalar1988error}. See \cite[Section 7]{bauschke2013rate} for a comparison between DRS and MAP for two subspaces.
\end{remark}

With this interpretation of MAP we can examine the inconsistent case,  $C_f \cap C_g = \emptyset$, from a different perspective than the current literature.  A part of the following result appeared in \cite[Theorem 4.8]{bauschke1994dykstra}.  In particular, if $x$ satisfies Equation~\eqref{cor:AP2:eq:gap}, then $P_{C_f}x - P_{C_g}x$ is the gap vector of \cite[Theorem 4.8]{bauschke1994dykstra}.
\begin{corollary}[Convergence of MAP: infeasible case]\label{cor:AP2}
Let $(z^j)_{j \geq 0}$ be generated by MAP, and suppose that $C_f \cap C_g = \emptyset$.  If there exists $x \in \cH$ such that
\begin{align}\label{cor:AP2:eq:gap}
x - P_{C_f}x = P_{C_g}x - x,
\end{align}
then $(z^j)_{j \geq 0}$ converges weakly to a point in the following set:
\begin{align}\label{eq:APfixedpoints}
\{ P_{C_f}x \mid x\in \cH, x - P_{C_f}x = P_{C_g}x - x\} \subseteq C_f,
\end{align}
with FPR rate $\|z^{k+1} - z^k\|^2 = o\left({1}/({k+1})\right)$. Furthermore, if $x$ satisfies Equation~\eqref{cor:AP2:eq:gap}, then
\begin{align}\label{cor:AP2:eq:sum}
\sum_{i=0}^\infty \left(\left\|\frac{1}{2}(z^i - P_{C_g}z^i) - (x - P_{C_g}x)\right\|^2 + \left\|\frac{1}{2}(P_{C_g}z^i - P_{C_f}P_{C_g}z^i) - (x - P_{C_f}x)\right\|^2 \right) < \infty.
\end{align}
In particular, the vector $P_{C_g}z^k - P_{C_f}P_{C_g}z^k$ strongly converges to the gap vector $P_{C_g} x - P_{C_f} x$, and  $$\min_{i = 0, \cdots, k}\left\{\|(P_{C_g}z^{i} - P_{C_f}z^{i}) - (P_{C_g} x - P_{C_f} x)\|^2\right\} = o(1/(k+1)).$$
\end{corollary}
\begin{proof}
In view of Proposition~\ref{prop:factsaboutdistancesquared}, the condition $x - P_{C_f}x = P_{C_g}x -x$ is equivalent to $ x\in \zer(\nabla d_{C_f}^2 + \nabla d_{C_g}^2)$.  The mapping $\TPRS^{{1}/{2}, {1}/{2}} = P_{C_f}P_{C_g}$ is the composition of $({1}/{2})$-averaged maps, and so it is $\alpha$-averaged for some $\alpha < 1$ \cite[Proposition 4.32]{bauschke2011convex}.  In addition, \cite[Theorem 1]{davis2014convergence} shows that the FPR satisfies $\|z^{k+1} - z^k\|^2 = o\left({1}/({k+1})\right)$. The set in Equation~\eqref{eq:APfixedpoints} is precisely the set of fixed points of $\TPRS$.  Therefore, weak convergence follows from \cite[Proposition 5.15]{bauschke2011convex}.  The sum in Equation~\eqref{cor:AP2:eq:sum} is exactly the sum of derivatives $\|\nabla d_{C_g}^2(x_g^k) - \nabla d_{C_g}^2(x)\|^2 + \|\nabla d_{C_f}^2(x_f^k) - \nabla d_{C_f}^2(x)\|^2$, and so it is finite by Proposition~\ref{prop:sumauxilliaryterms}.  Finally, strong convergence of $P_{C_g}z^k - P_{C_f}P_{C_g}z^k$ to the gap vector follows from the identity $x - P_{C_f}x = (1/2)(P_{C_g} x - P_{C_f} x)$. The rate is a consequence of~Fact~\ref{lem:sumsequence} and Equation~\eqref{cor:AP2:eq:sum}.
\qed\end{proof}

\begin{remark}\label{rem:gapvector}
Note that that the condition $x - P_{C_f}x = x - P_{C_g}x$ is equivalent to $\|P_{C_g}x - P_{C_f}x\|^2 = \min_{y \in \cH} (d_{C_f}^2(y) + d_{C_g}^2(y))=  \min_{x_f \in C_f, x_g \in C_g} \|x_g - x_f\|^2$.  See \cite[Fact 5.1]{bauschke1994dykstra} for conditions that guarantee the infimum is attained in Corollary~\ref{cor:AP2}.
\end{remark}

See Appendix~\ref{app:feasibilitymultiplesets} for the extension of the results of this section to finite collections of sets.

\section{From relaxed PRS to ADMM}\label{sec:DRSADMM}
The relaxed PRS algorithm can be applied to problem \eqref{eq:simplelinearconstrained}.  To this end we define the Lagrangian:
$$\cL_\gamma(x,y;w):=f(x)+g(y)-\dotp{w,Ax+By-b} + \frac{\gamma}{2}\|Ax + By - b\|^2.$$
Section~\ref{sec:DRSADMM} presents Algorithm \ref{alg:DRS} applied to the Lagrange dual of ~\eqref{eq:simplelinearconstrained}, which  reduces to the following algorithm:

\begin{algorithm}[H]
%
\begin{algorithmic}
\Require $w^{-1} \in \cH, x^{-1} = 0, y^{-1} = 0, \lambda_{-1} = \frac{1}{2}$, $\gamma > 0, (\lambda_j)_{j \geq 0} \subseteq (0, 1]$
\For{$k=-1,~0,\ldots$}
  \State $y^{k+1} = \argmin_{y} \cL_\gamma(x^k,y;w^k) + \gamma(2\lambda_k - 1) \dotp{By,  (Ax^{k} + By^{k} -b)}$\;
  \State $w^{k+1} = w^{k} - \gamma (Ax^{k} +  By^{k+1} - b) - \gamma(2\lambda_k - 1)(Ax^{k} + By^{k} - b)$\;
  \State $x^{k+1} = \argmin_{x} \cL_\gamma(x,y^{k+1};w^{k+1})$\;
\EndFor
\end{algorithmic}
\caption{{Relaxed alternating direction method of multipliers (relaxed ADMM)}}
\label{alg:ADMM}
\end{algorithm}
If $\lambda_k\equiv 1/2$, Algorithm \ref{alg:ADMM} recovers the standard ADMM.

It is well known that ADMM is equivalent to DRS applied to the Lagrange dual of Problem~\eqref{eq:simplelinearconstrained} \cite{gabay1983chapter}. Thus, if we let
\begin{align*}
d_f(w) := f^\ast(A^\ast w) && \mathrm{and} && d_g(w) := g^\ast(B^\ast w) - \dotp{w, b},
\end{align*}
then relaxed ADMM is equivalent to relaxed PRS applied to the following problem:
\begin{align}\label{eq:dualproblem2}
\Min_{w \in \cG} & \; d_f(w) + d_g(w).
\end{align}

{We make two assumptions regarding $d_f$ and $d_g$.
\begin{assump}[Solution existence]\label{assump:additivesubdual}
Functions $f, g : \cH \rightarrow (-\infty, \infty]$ satisfy
\begin{align}
\zer(\partial d_f + \partial d_g) \neq \emptyset.
\end{align}
\end{assump}
This is a restatement of Assumption~\ref{assump:additivesub}, which we have used in our analysis of the primal case.

\begin{assump}\label{assump:precomposgradient}
The following differentiation rule holds:
\begin{align*}
\partial d_f(x) = A^\ast \circ (\partial f^\ast) \circ A && \mathrm{and} && \partial d_g(x) = B^\ast \circ (\partial g^\ast) \circ B.
\end{align*}
\end{assump}
See \cite[Theorem 16.37]{bauschke2011convex} for conditions that imply this identity, of which the weakest are  $0 \in \mathrm{sri}(\range(A^\ast) - \dom(f^\ast))$ and $0 \in \mathrm{sri}(\range(B^\ast) - \dom(g^\ast))$, where $\mathrm{sri}$ is the strong relative interior of a convex set.} This assumption may seem strong, but it is standard in the analysis of ADMM because it implies the form in Proposition~\ref{prop:relaxedADMM}.

The next proposition shows how the strong convexity and the differentiability of a closed, proper, and convex function transfer to the dual function.
\begin{proposition}[Strong convexity and differentiability of the conjugate]\label{prop:strongconvextodual}
Suppose that $f : \cH \rightarrow (-\infty, \infty]$ is closed, proper, and convex. Then the following implications hold:
\begin{enumerate}
\item If $f$ is $\mu_f$-strongly convex, then $f^\ast$ is differentiable and $\nabla f$ is $({1}/{\mu_f})$-Lipschitz.
\item If $f$ is differentiable and $\nabla f$ is $({1}/{\beta})$-Lipschitz, then $f^\ast$ is $\beta$-strongly convex.
\end{enumerate}
\end{proposition}
\begin{proof}
See~\cite[Theorem 18.15]{bauschke2011convex}.
\qed\end{proof}
With Proposition~\ref{prop:strongconvextodual}, we can characterize the strong convexity and differentiability of the dual functions in terms of $A, B$ and $f$ and $g$. We first recall that a linear map $L : \cG \rightarrow \cG$ is \emph{$\alpha$-strongly monotone} if for all $x \in \cG$, the bound  $\dotp{Lx, x}_\cG \geq \alpha \|x\|_\cG^2$ holds.

\begin{proposition}[Strong convexity and differentiability of the dual]\label{prop:stronglymonotone}
The following implications hold:
\begin{enumerate}
\item If $\nabla f$, (respectively $\nabla g$), is $({1}/{\beta})$-Lipschitz and $AA^\ast$ (respectively $BB^\ast$) is $\alpha$-strongly monotone, then $d_f$ (respectively $d_g$) is $\alpha\beta$-strongly convex.
\item If $f$, (respectively $g$) is $\mu$-strongly convex, then $d_f$ (respectively $d_g$) is differentiable and $\nabla d_f$ (respectively $\nabla d_g$) is $({\|A\|^2}/{\mu})$ (respectively $({\|B\|^2}/{\mu})$)-Lipschitz.
\end{enumerate}
\end{proposition}
The proof of Proposition~\ref{prop:stronglymonotone} is straightforward, so we omit it. We note that $AA^\ast$ and $BB^\ast$ are always $0$-strongly monotone. Thus, we assume that $AA^\ast$ and $BB^\ast$ are $\alpha_A$ and $\alpha_B$-strongly monotone, respectively, while allowing the cases $\alpha_A = 0$ and $\alpha_B= 0$.  In addition, we use the convention that $\tnabla f$ and $\tnabla g$ are always $({1}/{\beta_f})$, and $({1}/{\beta_g})$-Lipschitz, respectively, by allowing the cases $\beta_f = 0$ and $\beta_g = 0$. We carry the following notation throughout the rest of Section~\ref{sec:DRSADMM}:
\begin{align}
\mu_{d_f} = \beta_f \alpha_A \geq 0 && \mathrm{and} && \mu_{d_g} = \beta_g \alpha_B \geq 0.
\end{align}
Thus, $d_f$ and $d_g$ are $\mu_{d_f}$ and $\mu_{d_g}$-strongly convex, respectively. Finally, we always assume that $f$ and $g$ are $\mu_f$ and $\mu_g$-strongly convex, respectively, by allowing $\mu_f=0$ and $\mu_g = 0$. We assume that $\|A\|\|B\| \neq 0$, and denote
\begin{align}
\beta_{d_f} = \frac{\mu_f}{\|A\|^2} \geq 0 && \mathrm{and} && \beta_{d_g} = \frac{\mu_g}{\|B\|^2} \geq 0.
\end{align}
If $\beta_{d_f}$ is strictly positive, then $d_f$ is differentiable and $\nabla d_f$ is $(1/\beta_f)$-Lipschitz. A similar result holds for $d_g$.

Now we apply Algorithm~\ref{alg:DRS} to the dual problem in Equation~\eqref{eq:dualproblem2}. Given $z^0 \in \cH$, Lemma~\ref{prop:DRSmainidentity} shows that we need to compute the following vectors for all $k\geq 0$:
\begin{align}\label{eq:DRSADMM}
\begin{cases}
w_{d_g}^k &= \prox_{\gamma d_g}(z^k); \\
w_{d_f}^k &= \prox_{\gamma d_f}(2w_{d_g}^k - z^k);  \\
z^{k+1} &= z^k + 2\lambda_k(w_{d_f}^k - w_{d_g}^k).
\end{cases}
\end{align}

%
%

A detailed proof of Proposition~\ref{prop:relaxedADMM} recently appeared in \cite[Proposition 11]{davis2014convergence}.
\begin{proposition}[Relaxed ADMM]\label{prop:relaxedADMM}
Let $z^0 \in \cG$, and let $(z^j)_{j \geq 0}$ be generated by the relaxed PRS algorithm applied to the dual formulation in Equation~(\ref{eq:dualproblem2}).  Choose $w_{d_g}^{-1} = z^0, x^{-1} = 0$ and $y^{-1} = 0$ and $\lambda_{-1} = {1}/{2}$. Then we have the following identities starting from $k = -1$:
\begin{align*}
y^{k+1} &= \argmin_{y \in \cH_2} g(y) - \dotp{w_{d_g}^{k},Ax^{k} +  By - b} + \frac{\gamma}{2} \|Ax^k + By - b + (2\lambda_k - 1)(Ax^{k} + By^{k} -b) \|^2; \\
w_{d_g}^{k+1} &= w_{d_g}^{k} - \gamma (Ax^{k} +  By^{k+1} - b) - \gamma(2\lambda_k - 1)(Ax^{k} + By^{k} - b); \\
x^{k+1} &= \argmin_{x \in \cH_1} f(x) - \dotp{w_{d_g}^{k+1}, Ax + By^{k+1} - b} + \frac{\gamma}{2} \|Ax + By^{k+1} - b\|^2; \\
w_{d_f}^{k+1} &= w_{d_g}^{k+1} - \gamma (Ax^{k+1} + By^{k+1} - b).
\end{align*}
\end{proposition}

\begin{remark}
Proposition~\ref{prop:relaxedADMM} proves that $w_{d_f}^{k+1} = w_{d_g}^{k+1} - \gamma (Ax^{k+1} + By^{k+1} - b)$. Recall that by Equation~(\ref{eq:DRSADMM}), $z^{k+1} - z^k = 2\lambda_k(w_{d_f}^{k} - w_{d_g}^{k})$.  Therefore, it follows that
\begin{align}\label{eq:ADMMfeasibilityFPR}
z^{k+1} - z^k &= -2\gamma \lambda_k( Ax^{k} + By^k - b).
\end{align}
\end{remark}

\begin{center}
\begin{table}
\centering
    \begin{tabular}{lll}
    \hline
    Function    & Primal subgradient  & Dual subgradient          \\ \hline\hline
    $g$ &  $\tnabla g(y^s) = B^\ast w_{d_g}^s$ & $\tnabla d_g(w_{d_g}^s) = By^s - b$ \\\hline
    $f$ & $\tnabla f(x^s) = A^\ast w_{d_f}^s$ &  $\tnabla d_f(w_{d_f}^s) = Ax^s$\\  \hline
    \end{tabular}
 \caption{The main subgradient identities used throughout Section~\ref{sec:DRSADMM}. The letter $s$ denotes a superscript (e.g. $s = k$ or $s = \ast$).  See \cite{davis2014convergence} for a proof.}\label{table:ADMMsubgradients}
\end{table}
\end{center}

\subsection{Converting dual inequalities to primal inequalities}\label{sec:convertinequalities}

The ADMM algorithm generates $5$ sequences of iterates:
\begin{align*}
(z^j)_{j \geq 0}, (w_{d_f}^j)_{j \geq 0}, \mbox{ and } (w_{d_g}^j)_{j \geq 0} \subseteq \cG && \mbox{and} &&  (x^j)_{j \geq 0} \subseteq \cH_1, (y^j)_{j \geq 0} \subseteq \cH_2.
\end{align*}
In this section we recall some inequalities, which were derived in \cite[Section 8.2]{davis2014convergence}, that relate these sequences to each other through the primal and dual objective functions.  
In the following propositions, $z^\ast$ will denote a fixed point of $\TPRS$. The point $w^\ast := \prox_{\gamma d_g}(z^\ast)$ is a minimizer of the dual problem in Equation~\eqref{eq:dualproblem2}.  Finally, we let $x^\ast$ and $y^\ast$ be defined as in Table~\ref{table:ADMMsubgradients}.


\begin{proposition}[ADMM primal upper fundamental inequality]\label{prop:ADMMupper}
For all $k \geq 0$, we have the bound
\begin{align*}
4\gamma \lambda_k (f(x^k) + g(y^k) &- f(x^\ast) - g(y^\ast))\\
&\leq \|z^k - (z^\ast - w^\ast)\|^2 - \|z^{k+1} - (z^\ast - w^\ast)\|^2  + \left(1- \frac{1}{\lambda_k} \right) \|z^{k} - z^{k+1}\|^2. \numberthis \label{prop:ADMMupper:eq:main}
\end{align*}
\end{proposition}

\begin{proposition}[ADMM primal lower fundamental inequality]\label{prop:ADMMlower}
For all $x \in \dom(f)$ and $y \in \dom(g)$, we have the bound:
\begin{align}\label{prop:ADMMlower:eq:main}
f(x) + g(y) - f(x^\ast) - g(y^\ast) &\geq \dotp{Ax + By - b, w^\ast}.
\end{align}
\end{proposition}

\subsection{Converting dual convergence rates to primal convergence rates}\label{sec:dualtoprimalrates}

In this section, we use the inequalities deduced in Section~\ref{sec:convertinequalities} and the convergence rates proved in previous sections to derive convergence rates for the primal objective error and strong convergence of various quantities that appear in ADMM.  In addition, we translate the results of the previous sections and use Proposition~\ref{prop:stronglymonotone} to state all theorems in terms of purely primal quantities.

We recall the definition of the two auxiliary terms (Equation~\eqref{eq:snotation}):
\begin{align}\label{eq:snotationdual}
S_{d_f}(w_{d_f}^k, w^\ast) &=  \max\left\{\frac{\beta_f\alpha_A}{2}\|w_{d_f}^k - w^\ast\|_{}^2, \frac{\mu_f}{2\|A\|^2}\|Ax^k - Ax^\ast\|_{}^2\right\}, \\
S_{d_g}(w_{d_g}^k, w^\ast) &=  \max\left\{\frac{\beta_g\alpha_B}{2}\|w_{d_g}^k - w^\ast\|_{}^2, \frac{\mu_g}{2\|B\|^2}\|By^k - By^\ast\|_{}^2\right\}.
\end{align}
This form readily follows from Table~\ref{table:ADMMsubgradients}.

The following is a direct translation of Theorem~\ref{prop:sumauxilliaryterms} to the current setting.  Note that any of the Lipschitz, strong convexity, and strong monotonicity constants may be zero.
\begin{theorem}[Primal differentiability and strong convexity]\label{prop:primaldifferentiability}
Suppose that $(z^j)_{j \geq 0}$ is generated by Algorithm~\ref{alg:ADMM}.  Then
\begin{enumerate}
\item \label{prop:sumauxilliaryterms:part:1} \textbf{Best iterate convergence:} If $(\lambda_j)_{j \geq 0}$ is bounded away from zero, then
$\min_{i=0, \cdots, k}\left\{S_{d_f}(w_{d_f}^{i}, w^\ast)\right\} = o\left(1/(k+1)\right)$ and $\min_{i=0, \cdots, k}\{S_{d_g}(w_{d_g}^{i}, w^\ast)\}= o\left(1/(k+1)\right).$
\item \label{prop:sumauxilliaryterms:part:2} \textbf{Ergodic convergence:} Let $\overline{w}_{d_f}^k = (1/\Lambda_k) \sum_{i=0}^k w_{d_f}^i$, let $\overline{w}_{d_g}^k = (1/\Lambda_k) \sum_{i=0}^k\lambda_i w_{d_g}^i$, let $\overline{x}^k = (1/\Lambda_k) \sum_{i=0}^k x^i$, and let $\overline{y}^k = (1/\Lambda_k) \sum_{i=0}^k\lambda_i y^i$. Then
\begin{align*}
\max\left\{ \beta_f\alpha_A\left\|\overline{w}_{d_f}^k - w^\ast\right\|^2, \frac{\mu_f}{\|A\|^2}\left\|A\overline{x}^k - Ax^\ast\right\|^2\right\} &+ \max\left\{\beta_g\alpha_B\left\|\overline{w}_{d_g}^k - w^\ast\right\|^2 + \frac{\mu_g}{\|B\|^2}\left\|B\overline{y}^k- By^\ast\right\|^2\right\} \\
 &\leq \frac{\|z^0 - z^\ast\|^2}{4 \gamma \Lambda_k}.
\end{align*}
\item  \label{prop:sumauxilliaryterms:part:3} \textbf{General convergence:} If $\underline{\tau} = \inf_{j \geq 0} \lambda_j(1-\lambda_j) > 0$, then
$S_{f}(w_{d_f}^{k}, w^\ast) + S_g(w_{d_g}^{k}, w^\ast) = o({1}/{\sqrt{k+1}})$.
\end{enumerate}
\end{theorem}

The following proposition deduces $o(1/(k+1))$ objective error convergence of standard ADMM whenever $g$ is strongly convex, and $\gamma$ is small enough.
\begin{theorem}[Strong convexity of $g$]\label{thm:strongconvexadmm}
Suppose that $g$ is $\mu_g$-strongly convex. Let $\lambda_k \equiv 1/2$, and let $\gamma < \kappa \beta =  \kappa \mu_g/\|B\|^2$ (see Theorem~\ref{thm:differentiableobjective}). Then for all $k \geq 1$, we have the constraint violations convergence rate:
\begin{align*}
\|Ax^k + By^k - b\|^2 \leq \frac{\beta^2\|w_{d_g}^0 - w^\ast\|^2}{\gamma^2k^2\left(1+\gamma/\beta\right)^2\left(\beta^2 -\gamma^2/\kappa^2\right)} && \mathrm{and} && \|Ax^k + By^k - b\|^2 = o\left(\frac{1}{k^2}\right).
\end{align*}
Moreover, the primal objective errors satisfy
\begin{align*}
 \frac{-\beta\|w^\ast\|\|w_{d_g}^0 - w^\ast\|}{\gamma k\left(1+{\gamma}/{\beta}\right)\sqrt{\left(\beta^2 -{\gamma^2}/{\kappa^2}\right)}} \leq f(x^k) + g(y^k) - f(x^\ast) - g(y^\ast) &\leq\frac{\beta\|w_{d_g}^0 - w^\ast\| \left(\|z^0 - z^\ast\| + \|w^\ast\|\right)}{\gamma k \left(1+{\gamma}/{\beta}\right)\sqrt{\left(\beta^2 -{\gamma^2}/{\kappa^2}\right)}},
\end{align*}
and $|f(x^k) + g(y^k) - f(x^\ast) - g(y^\ast)| = o\left(1/k\right).$
\end{theorem}
\begin{proof}
The constraint violations rate follows from the identity $z^{k+1} - z^k = -\gamma ( Ax^{k} + By^k - b)$ (Equation~\eqref{eq:ADMMfeasibilityFPR}) and the FPR convergence rate in Theorem~\ref{thm:differentiableFPR}.

The lower bound follows from the lower fundamental inequality in Proposition~\ref{prop:ADMMlower} and the FPR convergence rate in Theorem~\ref{thm:differentiableFPR}:
\begin{align*}
f(x^k) + g(y^k) - f(x^\ast) - g(y^\ast) \stackrel{\eqref{prop:ADMMlower:eq:main}}{\geq} \dotp{Ax^k + By^k - b, w^\ast} &\stackrel{\eqref{eq:ADMMfeasibilityFPR}}{\geq} -\frac{1}{\gamma}\|z^k - z^{k+1}\|\| w^\ast\| \\
&\stackrel{\eqref{eq:differentiableFPR}}{\geq} - \frac{-\beta\|w^\ast\|\|w_{d_g}^0 - w^\ast\|}{\gamma k\left(1+{\gamma}/{\beta}\right)\sqrt{\left(\beta^2 -{\gamma^2}/{\kappa^2}\right)}}.
\end{align*}

Part~\ref{fact:averagedconvergence:eq:mono} of Fact~\ref{fact:averagedconvergence} bounds the norm: $\|z^{k+1} - (z^\ast - w^\ast)\| \leq \|z^{k+1} - z^\ast\| + \|w^\ast\| \leq \|z^0 - z^\ast\| + \|w^\ast\|$. Therefore, the upper bound follows from the upper fundamental inequality in Proposition~\ref{prop:ADMMupper} and the FPR convergence rate in Theorem~\ref{thm:differentiableFPR}:
\begin{align*}
 f(x^k) + g(y^k) - f(x^\ast) - g(y^\ast)
&\stackrel{\eqref{prop:ADMMupper:eq:main}}{\leq} \frac{1}{2\gamma}\left(\|z^k - (z^\ast - w^\ast)\|^2 - \|z^{k+1} - (z^\ast - w^\ast)\|^2  - \|z^{k} - z^{k+1}\|^2\right) \\\
&\stackrel{\eqref{eq:cosinerule}}{\leq} \frac{1}{\gamma}\dotp{ z^{k+1} - (z^\ast - w^\ast), z^k - z^{k+1}} \stackrel{\eqref{eq:differentiableFPR}}{\leq} \frac{\beta\|w_{d_g}^0 - w^\ast\| \left(\|z^0 - z^\ast\| + \|w^\ast\|\right)}{\gamma k \left(1+{\gamma}/{\beta}\right)\sqrt{\left(\beta^2 -{\gamma^2}/{\kappa^2}\right)}}.
\end{align*}
The little $o$-rate follows because, as the above equations have shown, the objective error is upper and lower bounded by a multiple of the square root of the FPR, which has convergence rate $o(1/k)$ by Theorem~\ref{thm:differentiableFPR}.
\qed\end{proof}

It would be nice to prove a convergence rate for the ``best iterate" of the sequence of primal objective errors in the style of Theorem~\ref{thm:lipschitzbest}.  Unfortunately the fundamental inequalities we developed in Section~\ref{sec:convertinequalities} do not immediately imply such a rate.

Now we shift our focus to linear convergence. The following proposition is a direct translation of the main results of Section~\ref{sec:linearconvergence} to the current setting. The interested reader is encouraged to read Appendix~\ref{app:linearconvergenceimplies} to see how the following rates imply convergence rates for the primal and dual objective, and feasibility errors.
\begin{theorem}[Linear convergence of Relaxed ADMM]\label{thm:admmlinearconvergence}
The following are true:
\begin{enumerate}
\item \label{thm:admmlinearconvergence:part:gregular} If $\mu_g \beta_g \alpha_B > 0$, then $(z^j)_{j \geq 0}$ converges linearly and
\begin{align*}
\|z^{k+1} - z^\ast\|^2 \leq \left(1 -\frac {4\gamma\lambda_k\beta_g\alpha_B}{(1 + {\gamma\|B\|^2}/{\mu_g})^2}\right)^{{1}/{2}}\|z^k - z^\ast\|^2.
\end{align*}
\item \label{thm:admmlinearconvergence:part:fregular} If $\mu_f\beta_f \alpha_A > 0$, then $(z^j)_{j \geq 0}$ converges linearly and
\begin{align*}
\|z^{k+1} - z^\ast\|^2 \leq \left(1 -   \frac{\lambda_k\min\left\{ {4\gamma  \beta_f \alpha_A}/{\left(1+{\|A\|^2}/{\mu_f}\right)^2}, (1 - \lambda_k)\right\}}{2}\right)^{{1}/{2}}\|z^k - z^\ast\|^2.
\end{align*}
\item \label{thm:admmlinearconvergence:part:fgregular}  If $\mu_f \beta_g \alpha_B > 0$, then $(z^j)_{j \geq 0}$ converges linearly and
\begin{align*}
\|z^{k+1} - z^\ast\|^2 \leq \left(1 -   \frac{4\lambda_k\min\{ \gamma  \beta_g\alpha_B, { \mu_f}/({\|A\|^2\gamma}), (1 - \lambda_k)\}}{3}\right)^{{1}/{2}}\|z^k - z^\ast\|^2.
\end{align*}
\item\label{thm:admmlinearconvergence:part:gfregular}  If $\mu_g \beta_f \alpha_A > 0$, then $(z^j)_{j \geq 0}$ converges linearly and
\begin{align*}
\|z^{k+1} - z^\ast\|^2 \leq \left(1 -   \frac{4\lambda_k\min\{ \gamma  \beta_f\alpha_A, { \mu_g}/({\|B\|^2\gamma}), (1 - \lambda_k)\}}{3}\right)^{{1}/{2}}\|z^k - z^\ast\|^2.
\end{align*}
\end{enumerate}
\end{theorem}

We can apply Proposition~\ref{prop:linearconvergenceimpliesADMM} to any of the scenarios that appear in Theorem~\ref{thm:admmlinearconvergence} and deduce the rate of linear convergence of the objective error and constraint violations.  We leave this application to the reader.

Linear convergence of ADMM has been deduced in a variety of scenarios. In \cite{denglinear2012}, the authors prove the linear convergence (in finite dimensions) of a generalized form of ADMM, which allows the possibility of adding proximal terms to the alternating minimization steps that appear in Algorithm~\ref{alg:ADMM}. The four scenarios that appear in \cite[Table 1.1]{denglinear2012}) have some overlap with our results. In the standard version of ADMM, (with no relaxation or extra proximal terms), scenarios 1 and 2 in~\cite[Table 1.1]{denglinear2012} are the finite-dimensional analogues of Part~\ref{thm:admmlinearconvergence:part:gregular} of Theorem~\ref{sec:linearconvergence}. Scenarios 3 and 4 in~\cite[Table 1.1]{denglinear2012} are not covered by our analysis because they require that we treat the structure of $A$ and $B$ more carefully than we have in this section.  In addition,  Parts~\ref{thm:admmlinearconvergence:part:fregular},~\ref{thm:admmlinearconvergence:part:fgregular}, and~\ref{thm:admmlinearconvergence:part:gfregular} of Theorem~\ref{thm:admmlinearconvergence} are not discussed in \cite{denglinear2012}. Finally, we note that this paper and \cite{denglinear2012} use the opposite update orders in ADMM. They generally lead to different sequences except when at least one of $f$ and $g$ is quadratic \cite{yanyin2014}. Therefore, when comparing the results between the two papers, one must switch $f$ and $g$, as well as $A$ and $B$.

\section{Examples}

In this section, we apply DRS and ADMM to concrete problems and explicitly bound the associated objective errors and FPR with the convergence rates that we derived in the previous sections.

\subsection{Feasibility problems}\label{sec:feasibility}

Suppose that $C_f$ and $C_g$ are closed convex subsets of $\cH$ with nonempty intersection.  The goal of the feasibility problem is the find a point in the intersection of $C_f$ and $C_g$. In this section, we present a comparison between MAP and the relaxed PRS algorithm.

\subsubsection{Linear convergence}
Section~\ref{sec:feasibilityniceintersection} shows that relaxed PRS applied to $f = d_{C_f}^2$ and $g = d_{C_g}^2$ converges linearly whenever $C_f$ and $C_g$ have a sufficiently nice intersection.  In addition, \cite{bauschke2014linear} and \cite{phan2014linear} have recently shown that one can achieve linear convergence under the same regularity assumptions on $C_f \cap C_g$ when $f= \iota_{C_f}$ and $g = \iota_{C_f}$. We refer to \cite[Fact 5.8]{bauschke2014linear} for an extensive list of conditions that guarantee (bounded) linear regularity of $\{C_1, C_2\}$. For the readers convenience, we list a few important examples:
\begin{enumerate}
\item {\bf Subspaces:} If $C_f^\perp + C_g^\perp$ is closed, then $\{C_f, C_g\}$ is linearly regular.
\item {\bf Polyhedron:} If $C_f \cap C_g \neq \emptyset$, then $\{C_f, C_g\}$ is linearly regular.
\item {\bf Standard constraint qualification:} If the relative interiors of $C_f$ and $C_g$ intersect, then $\{C_f,C_g\}$ is boundedly linearly regular.
\end{enumerate}

\subsubsection{General convergence}\label{sec:feasibilitygeneralconvergence}

In general, we cannot expect linear convergence of relaxed PRS algorithm for the feasibility problem.  Indeed, \cite[Theorem 9]{davis2014convergence} constructs a DRS iteration that converges in norm but does so \emph{arbitrarily slowly}.  A similar result holds for MAP~\cite{bauschke2009characterizing}.  Thus, in  \cite{davis2014convergence} the authors focused on other measures of convergence, namely \emph{FPR} and \emph{objective error} rate.  The following discussion will utilize the results of \cite{davis2014convergence} to compare the relaxed PRS and MAP algorithms in the absence of regularity.

Let $\iota_{C_f}$ and $\iota_{C_g}$ be the indicator functions of $C_f$ and $C_g$. Then $x \in C_f \cap C_g$, if, and only if, $\iota_{C_f}(x) + \iota_{C_g}(x) = 0$, and the sum is infinite otherwise.  Thus, a point is in the intersection of $C_f$ and $C_g$ if, and only if, it is the minimizer of the following problem:
\begin{align}\label{sec:feasibility:eq:chiminimize}
\Min_{x \in \cH} \iota_{C_f}(x) + \iota_{C_g}(x).
\end{align}
The relaxed PRS algorithm applied to $f = \iota_{C_f}$ and $g = \iota_{C_g}$ has the following form: given an initial point $z^0 \in \cH$, for all $k \geq 0$, define
\begin{align}\label{sec:feasibility:eq:DRSchi}
\begin{cases}
x_g^k &= P_{C_g}(z^k); \\
x_f^k &= P_{C_f}(2x_g^k - z^k); \\
z^{k+1} &= z^k + 2\lambda_k(x_f^k - x_g^k).
\end{cases}
\end{align}

In general, the functions $f$ and $g$ are neither differentiable nor strongly convex. Furthermore, they only take on the values $0$ and $\infty$. Thus, we will only discuss FPR convergence rates of relaxed PRS. The FPR identity $x_{f}^k - x_g^k = \frac{1}{2\lambda_k}(z^{k+1} - z^k)$ shows that after $k$ iterations
\begin{align}\label{eq:feasibilitybounddistancenonergodic}
\max\{ d^2_{C_g}(x_f^k), d^2_{C_f}(x_g^k)\} \leq \|x_f^k - x_g^k\|^2 &\stackrel{\eqref{cor:DRSaveragedconvergence:eq:main}}{=} o\left(\frac{1}{k+1}\right).
\end{align}
By the convexity of $C_f$ and $C_g$, the ergodic iterates of relaxed PRS satisfy $\overline{x}_f^k = ({1}/{\Lambda_k})\sum_{i=0}^k \lambda_i x_f^i \in C_f$ and $\overline{x}_g^k = ({1}/{\Lambda_k})\sum_{i=0}^k \lambda_i x_g^i \in C_g$. Thus, \cite[Theorem 6]{davis2014convergence} implies the improved bound
\begin{align}\label{eq:drsergodicdistancebound}
\max\{ d^2_{C_g}(\overline{x}_f^k), d^2_{C_f}(\overline{x}_g^k)\} \leq \|\overline{x}_f^k - \overline{x}_g^k\|^2 &= O\left(\frac{1}{\Lambda_k^2}\right),
\end{align}
which is optimal by \cite[Proposition 7]{davis2014convergence}.  Therefore, after $k$ iterations the relaxed PRS algorithm produces a point in each set with distance of order at most $O({1}/{\Lambda_k})$ from each other.

We now shift our focus to the MAP algorithm. First we replace both of the indicator functions with the squared distance functions: $f= \min_{y \in C_f} \|x - y\|^2$ and $g(x) = \min_{y \in C_g} \|x-y\|^2.$ Now recall that $f$ and $g$ are differentiable, the gradient $\nabla g$ is $2$-Lipschitz continuous \cite[Corollary 12.30]{bauschke2011convex}, and relaxed PRS takes the form in Equation~\eqref{eq:DRSfeasibilitygamma}. Specializing to $\gamma = {1}/{2}$ and $\lambda_k \equiv 1$ yields the MAP algorithm (Corollary~\ref{cor:AP}).

In this algorithm, the main MAP sequence satisfies $(z^j)_{j \geq 1} \subseteq C_f$, while the auxiliary sequences $(x_f^j)_{j \geq 0}$ and $(x_g^j)_{j \geq 0}$ are not necessarily elements $C_f$ or $C_g$.  Therefore, the MAP FPR rate is less useful for estimating distances of the current iterates to $C_f$ and $C_g$ than it is in the relaxed PRS algorithm (See Equation~\eqref{eq:feasibilitybounddistancenonergodic}). Although $\lambda_k \equiv 1$, the map $P_{C_f} P_{C_g}$ is $\alpha$-averaged for some $\alpha < 1$, and, hence, we can still estimate $\|z^{k+1} - z^k\|^2 = o({1}/{(k+1)})$  (Corollary~\ref{cor:AP2}).

The ergodic convergence rate in \cite[Theorem 6]{davis2014convergence} (where we use the identity $d_{C_g}(x_g^k) = ({1}/{2})d_{C_g}(z^k)$ and Jensen's inequality) shows that
\begin{align}
d^2_{C_g}\left(\frac{1}{k+1}\sum_{i=0}^k z^i\right) \leq \frac{2}{k+1}\sum_{i=0}^k d_{C_g}^2(x_g^k) = O\left(\frac{1}{k+1}\right).
\end{align}
Thus, if we choose $z^0 \in C_f$, the ergodic iterate $({1}/{(k+1)})\sum_{i=0}^k z^i $ is an element of $C_f$ and we can bound its distance from $C_g$.  Note that this rate is strictly slower than the rate in Equation~\eqref{eq:drsergodicdistancebound}.

Although $d_{C_f}^2$ and $d_{C_g}^2$ are differentiable (Proposition~\ref{prop:factsaboutdistancesquared}), we cannot apply the results of Section~\ref{sec:lipschitzderivatives} to MAP because they require that $(\lambda_j)_{j \geq 0} \subseteq (0, 1)$. Therefore, we cannot use the regularity of $d_{C_f}^2$ and $d_{C_g}^2$ to deduce faster convergence of the AP algorithm.

This discussion shows that the convergence rates predicted in \cite{davis2014convergence} for relaxed PRS, which are known to be optimal, are faster than those predicted for MAP.  When $C_f$ and $C_g$ intersect nicely (Section~\ref{sec:feasibilityniceintersection}), the rate predicted for MAP is faster  (See Corollary~\ref{cor:AP}). In \cite[Section 8]{bauschke2013rate} a similar phenomenon is observed for the case of intersecting subspaces: DRS is faster than MAP for problems with nonregular intersection.  It would be highly satisfying to characterize this phenomenon in general.


\subsection{Parallelized model fitting and classification}\label{sec:modelfiting}
The following scenario appears in \cite[Chapter 8]{boyd2011distributed}. Consider the model fitting problem: Let $M : \vR^n \rightarrow \vR^m$ be a \emph{feature matrix}, let $b \in \vR^m$, be the  \emph{output} vector, let $l$ be a \emph{loss function} and let $r$ be a \emph{regularization function}. The goal of the \emph{model fitting problem} is to
\begin{align}\label{sec:modelfiting:eq:problem}
\Min_{x\in \vR^n } \; l(Mx -b) + r(x).
\end{align}
The function $l$ is used to enforce the constraint $Mx = b + \nu$ up to some noise $\nu$ in the measurement, while $r$ enforces the \emph{regularity} of $x$ by incorporating \emph{prior knowledge} of the form of the solution.

In this section, we present one way to split Equation~\eqref{sec:modelfiting:eq:problem}. Our discussion extends the one given in~\cite[Section 9.2]{davis2014convergence}, where only convexity of $l$ and $r$ is assumed.

\subsubsection{Auxiliary variable}
We can split Equation~\eqref{sec:modelfiting:eq:problem} by defining an auxiliary variable for $Mx - b$:
\begin{align*}
\Min_{x \in \vR^m, y \in\vR^n} & \; l\left(y \right) + r(x) \\
\mbox{subject to }  & \; Mx - y = b.\numberthis \label{sec:modelfiting:eq:pullAout}
\end{align*}
We will now analyze the convergence rates predicted in Section~\ref{sec:dualtoprimalrates} for ADMM applied to Problem~\eqref{sec:modelfiting:eq:pullAout}. Our most general convergence result applies to the auxiliary terms:
\begin{align*}
S_{d_r}(w_{d_r}^k, w^\ast) &=  \max\left\{\frac{\beta_r\alpha_M}{2}\|w_{d_r}^k - w^\ast\|_{\vR^m}^2, \frac{\mu_r}{2\|M\|^2}\|Mx^k - Mx^\ast\|_{\vR^m}^2\right\}, \\
S_{d_l}(w_{d_l}^k, w^\ast) &=  \max\left\{\frac{\beta_l}{2}\|w_{d_l}^k - w^\ast\|_{\vR^m}^2, \frac{\mu_l}{2}\|y^k - y^\ast\|_{\vR^n}^2\right\}.
\end{align*}
Theorem~\ref{prop:primaldifferentiability} shows that the best auxiliary term converges with rate $o(1/(k+1))$, the ergodic auxiliary term converges with rate $O(1/\Lambda_k)$, and the entire sequence of auxiliary terms converges with rate $o(1/\sqrt{k+1})$.

Now suppose that $\mu_l > 0$.  Then we can bound the distance of $y^k$ to the optimal point $y^\ast: = Mx^\ast - b$: $$\|\overline{y}^k - y^\ast\|^2 = O\left(\frac{1}{\Lambda_k^2}\right).$$ Now let $f = r$, let $g = l$, let $A = M$, and let $B = -I_{\cR^m}$. If $\gamma <\kappa \mu_l$, then Theorem~\ref{thm:strongconvexadmm} bounds the primal objective error and the FPR:
\begin{align*}
|l(y^k) + r(x^k) - l(Mx^\ast - b) - r(x^\ast)| = o\left(\frac{1}{{k+1}}\right) && \mathrm{and} &&  \|Mx^k - b- y^k\|^2 = o\left(\frac{1}{({k+1})^2}\right).
\end{align*}
In particular, if $l$ is Lipschitz, then $|l(y^k) - l(Mx^k - b)| = o\left(1/({k+1})\right)$. Thus, we have
\begin{align*}
0 \leq l(Mx^k - b) + r(x^k) - l(Mx^\ast - b) - r(x^\ast) = o\left(\frac{1}{{k+1}}\right).
\end{align*}
A similar result holds if $r$ is strongly convex and we assign $g= r$ and $f = l$, etc.

We can improve the above sublinear rate to a linear rate in any of the following cases (Theorem~\ref{thm:admmlinearconvergence}):
\begin{itemize}
\item $r$ is differentiable and strongly convex and $MM^\ast$ is strongly monotone;
\item $l$ is differentiable and strongly convex;
\item $r$ is differentiable, $MM^\ast$ is strongly monotone, and $l$ is strongly convex;
\item $r$ is strongly convex and $l$ is differentiable.
\end{itemize}

\section{Conclusion}
In this paper, we provided a comprehensive convergence rate analysis of relaxed PRS and ADMM under various regularity assumptions. By appealing to the examples developed in \cite{davis2014convergence}, we showed that several of the convergence rates cannot be improved.  All results follow from some combination of a lemma that deduces convergence rates of summable monotonic sequences (Lemma~\ref{lem:sumsequence}), a simple diagram (Figure~\ref{fig:DRSTR}), and fundamental inequalities (Propositions~\ref{prop:DRSupper},~\ref{prop:DRSlower},~\ref{prop:DRSupperfglipschitz},  and~\ref{prop:fundamentaldiff}) that relate the FPR to the objective error of the relaxed PRS algorithm.  Thus, together with~\cite{davis2014convergence}, we have developed a comprehensive convergence rate of the relaxed PRS and ADMM algorithms under the standard regularity assumptions in convex optimization.

\begin{acknowledgements}D. Davis' work is partially supported by NSF GRFP grant DGE-0707424. W. Yin's work is partially supported by NSF grants DMS-0748839 and DMS-1317602.
\end{acknowledgements}

\bibliographystyle{spmpsci}
\bibliography{bibliography}

\begin{thebibliography}{10}
\providecommand{\url}[1]{{#1}}
\providecommand{\urlprefix}{URL }
\expandafter\ifx\csname urlstyle\endcsname\relax
  \providecommand{\doi}[1]{DOI~\discretionary{}{}{}#1}\else
  \providecommand{\doi}{DOI~\discretionary{}{}{}\begingroup
  \urlstyle{rm}\Url}\fi

\bibitem{aronszajn1950theory}
Aronszajn, N.: Theory of reproducing kernels.
\newblock Transactions of the American mathematical society pp. 337--404 (1950)

\bibitem{baillon1977quelques}
Baillon, J.B., Haddad, G.: Quelques propri{\'e}t{\'e}s des op{\'e}rateurs
  angle-born{\'e}s et $n$-cycliquement monotones.
\newblock Israel Journal of Mathematics \textbf{26}(2), 137--150 (1977)

\bibitem{bauschke1996projectionthesis}
Bauschke, H.H.: Projection algorithms and monotone operators.
\newblock Ph.D. thesis, Theses (Dept. of Mathematics and Statistics)/Simon
  Fraser University (1996)

\bibitem{bauschke2013rate}
Bauschke, H.H., {Bello Cruz}, J.Y., Nghia, T.T.{\relax A}., Phan, H.M., Wang,
  X.: The rate of linear convergence of the {Douglas-Rachford} algorithm for
  subspaces is the cosine of the friedrichs angle.
\newblock Journal of Approximation Theory \textbf{185}(0), 63--79 (2014).
\newblock \doi{http://dx.doi.org/10.1016/j.jat.2014.06.002}

\bibitem{bauschke1993convergence}
Bauschke, H.H., Borwein, J.M.: On the convergence of {von Neumann's}
  alternating projection algorithm for two sets.
\newblock Set-Valued Analysis \textbf{1}(2), 185--212 (1993)

\bibitem{bauschke1994dykstra}
Bauschke, H.H., Borwein, J.M.: Dykstra's alternating projection algorithm for
  two sets.
\newblock Journal of Approximation Theory \textbf{79}(3), 418--443 (1994)

\bibitem{bauschke1996projection}
Bauschke, H.H., Borwein, J.M.: On projection algorithms for solving convex
  feasibility problems.
\newblock SIAM review \textbf{38}(3), 367--426 (1996)

\bibitem{bauschke1999strong}
Bauschke, H.H., Borwein, J.M., Li, W.: Strong conical hull intersection
  property, bounded linear regularity, {JamesonÕs} property (g), and error
  bounds in convex optimization.
\newblock Mathematical Programming \textbf{86}(1), 135--160 (1999)

\bibitem{bauschke2011convex}
Bauschke, H.H., Combettes, P.L.: Convex analysis and monotone operator theory
  in Hilbert spaces.
\newblock Springer (2011)

\bibitem{bauschke2009characterizing}
Bauschke, H.H., Deutsch, F., Hundal, H.: Characterizing arbitrarily slow
  convergence in the method of alternating projections.
\newblock International Transactions in Operational Research \textbf{16}(4),
  413--425 (2009)

\bibitem{bauschke2014linear}
Bauschke, H.H., Noll, D., Phan, H.M.: Linear and strong convergence of
  algorithms involving averaged nonexpansive operators.
\newblock arXiv preprint arXiv:1402.5460  (2014)

\bibitem{bauschke2013method}
Bauschke, H.H., Phan, H.M., Wang, X.: The method of alternating relaxed
  projections for two nonconvex sets.
\newblock Vietnam Journal of Mathematics pp. 1--30 (2013)

\bibitem{bertsekas2011incremental}
Bertsekas, D.P.: Incremental gradient, subgradient, and proximal methods for
  convex optimization: A survey.
\newblock Optimization for Machine Learning \textbf{2010}, 1--38 (2011)

\bibitem{bertsekas1989parallel}
Bertsekas, D.P., Tsitsiklis, J.N.: Parallel and distributed computation:
  numerical methods.
\newblock Prentice-Hall, Inc. (1989)

\bibitem{boyd2011distributed}
Boyd, S., Parikh, N., Chu, E., Peleato, B., Eckstein, J.: Distributed
  optimization and statistical learning via the alternating direction method of
  multipliers.
\newblock Foundations and Trends in Machine Learning \textbf{3}(1), 1--122
  (2011)

\bibitem{combettes2011proximal}
Combettes, P.L., Pesquet, J.C.: Proximal splitting methods in signal
  processing.
\newblock In: Fixed-point algorithms for inverse problems in science and
  engineering, pp. 185--212. Springer (2011)

\bibitem{davis2014convergence}
Davis, D., Yin, W.: Convergence rate analysis of several splitting schemes.
\newblock arXiv preprint arXiv:1406.4834  (2014)

\bibitem{denglinear2012}
Deng, W., Yin, W.: On the global and linear convergence of the generalized
  alternating direction method of multipliers.
\newblock Rice CAAM technical report 12-14  (2012)

\bibitem{deutsch2008rate}
Deutsch, F., Hundal, H.: The rate of convergence for the cyclic projections
  algorithm iii: Regularity of convex sets.
\newblock Journal of Approximation Theory \textbf{155}(2), 155--184 (2008)

\bibitem{gabay1983chapter}
Gabay, D.: Chapter ix applications of the method of multipliers to variational
  inequalities.
\newblock Studies in mathematics and its applications \textbf{15}, 299--331
  (1983)

\bibitem{gabay1976dual}
Gabay, D., Mercier, B.: A dual algorithm for the solution of nonlinear
  variational problems via finite element approximation.
\newblock Computers \& Mathematics with Applications \textbf{2}(1), 17--40
  (1976)

\bibitem{GlowinskiADMM}
Glowinski, R., Marrocco, A.: Sur l'approximation, par {\'e}l{\'e}ments finis
  d'ordre un, et la r{\'e}solution, par p{\'e}nalisation-dualit{\'e} d'une
  classe de probl{\`e}mes de {Dirichlet} nonlin{\'e}aires.
\newblock Rev. Francaise dAut. Inf. Rech. Oper \textbf{R-2}, 41--76 (1975)

\bibitem{goemans1995improved}
Goemans, M.X., Williamson, D.P.: Improved approximation algorithms for maximum
  cut and satisfiability problems using semidefinite programming.
\newblock Journal of the ACM (JACM) \textbf{42}(6), 1115--1145 (1995)

\bibitem{goldstein2009split}
Goldstein, T., Osher, S.: The split {Bregman} method for {L1}-regularized
  problems.
\newblock SIAM Journal on Imaging Sciences \textbf{2}(2), 323--343 (2009)

\bibitem{kayalar1988error}
Kayalar, S., Weinert, H.L.: Error bounds for the method of alternating
  projections.
\newblock Mathematics of Control, Signals and Systems \textbf{1}(1), 43--59
  (1988)

\bibitem{lions1979splitting}
Lions, P.L., Mercier, B.: Splitting algorithms for the sum of two nonlinear
  operators.
\newblock SIAM Journal on Numerical Analysis \textbf{16}(6), 964--979 (1979)

\bibitem{luke2008finding}
Luke, R.D.: Finding best approximation pairs relative to a convex and
  prox-regular set in a {Hilbert} space.
\newblock SIAM Journal on Optimization \textbf{19}(2), 714--739 (2008)

\bibitem{o2013operator}
O'Donoghue, B., Chu, E., Parikh, N., Boyd, S.: Operator splitting for conic
  optimization via homogeneous self-dual embedding.
\newblock arXiv preprint arXiv:1312.3039  (2013)

\bibitem{phan2014linear}
Phan, H.M.: Linear convergence of the {Douglas-Rachford} method for two closed
  sets.
\newblock arXiv preprint arXiv:1401.6509  (2014)

\bibitem{shi2013linear}
Shi, W., Ling, Q., Yuan, K., Wu, G., Yin, W.: On the linear convergence of the
  {ADMM} in decentralized consensus optimization.
\newblock arXiv preprint arXiv:1307.5561  (2013)

\bibitem{wei2012distributed}
Wei, E., Ozdaglar, A.: Distributed alternating direction method of multipliers.
\newblock In: Decision and Control (CDC), 2012 IEEE 51st Annual Conference on,
  pp. 5445--5450. IEEE (2012)

\bibitem{yanyin2014}
Yan, M., Yin, W.: Self equivalence of the alternating direction method of
  multipliers.
\newblock UCLA CAM Report \textbf{14-59} (2014)

\bibitem{ye1994nl}
Ye, Y., Todd, M.J., Mizuno, S.: An {$O(\sqrt{n}L)$}-iteration homogeneous and
  self-dual linear programming algorithm.
\newblock Mathematics of Operations Research \textbf{19}(1), 53--67 (1994)

\end{thebibliography}

\appendix
\section*{Appendices}
\addcontentsline{toc}{section}{Appendices}
\renewcommand{\thesection}{\Alph{section}}
\renewcommand{\theequation}{\Alph{section}.\arabic{equation}}
\section{Technical results from Section~\ref{sec:lipschitzderivatives:part:general}}\label{app:lipschitzderivatives:part:general}

The following Theorem will be used several times throughout our analysis.
\begin{theorem}[Descent theorem/Baillon-Haddad]\label{thm:descent}
Suppose that $g : \cH \rightarrow (-\infty, \infty]$ is closed, proper, convex, and differentiable.  If $\nabla g$ is $({1}/{\beta})$-Lipschitz, then for all $x, y \in \cH$, we have the upper bound
\begin{align}\label{eq:lipschitzderivative}
g(x) &\leq g(y) + \dotp{ x- y, \nabla g(y)} + \frac{1}{2\beta} \|x - y\|^2,
\end{align}
and the cocoercive inequality
\begin{align}\label{eq:baillon}
\beta\|\nabla g(x) - \nabla g(y)\|^2 &\leq \dotp{x - y, \nabla g(x) - \nabla g(y)}.
\end{align}
\end{theorem}
\begin{proof}
See \cite[Theorem 18.15(iii)]{bauschke2011convex} for Equation~\eqref{eq:lipschitzderivative}, and~\cite{baillon1977quelques} for Equation~\eqref{eq:baillon}.
\qed\end{proof}

\begin{proof}[Proof of Proposition~\ref{prop:DRSupperfglipschitz}]
Because $\nabla f$ is $(1/\beta)$-Lipschitz, we have
\begin{align*}
f(x_g) &\stackrel{\eqref{eq:lipschitzderivative}}{\leq} f(x_f) + \dotp{x_g - x_f, \nabla f(x_f)} + \frac{1}{2\beta}\|x_g - x_f\|^2,\numberthis \label{eq:fxgupperbound}\\
S_f(x_f,x^*)&\stackrel{\eqref{eq:strongconvexandlipschitzlowerbound}}{\ge} \frac{\beta}{2}\|\nabla f(x_f) - \nabla f(x^\ast)\|^2\numberthis \label{eq:Sflowerbound}.
\end{align*}
We now derive some identities that will be used below to bound $f(x_g) + g(x_g) - f(x^*) - g(x^*)$.
By applying the identity $z^\ast - x^\ast = \gamma\tnabla g(x^\ast) = - \gamma\nabla f(x^\ast)$ (Equation~\eqref{eq:gradoptimality}), the cosine rule \eqref{eq:cosinerule}, and Equation \eqref{eq:DRSmainidentity2} multiple times, we have
\begin{align*}
2\dotp{z - z^{+}, z^\ast - x^\ast} &+4\gamma \lambda \dotp{x_g - x_f, \nabla f(x_f)} = 4\gamma \lambda \dotp{x_g - x_f, \nabla f(x_f)-\nabla f(x^\ast)} \\
&= 4\lambda \dotp{ \gamma \tnabla g(x_g) +\gamma \nabla f(x_f), \gamma \nabla f(x_f)-\gamma \nabla f(x^\ast)}  \\
&= 2\lambda \left( \|x_f - x_g\|^2 + \|\gamma\nabla f(x_f) - \gamma\nabla f(x^\ast)\|^2-\|\gamma \tnabla g(x_g) - \gamma \tnabla g(x^\ast) \|^2\right).
\end{align*}
By Equation \eqref{eq:DRSmainidentity2}
$
\left(1-\frac{1}{\lambda}\right) \|z - z^+\|^2 + 2\lambda\left(\frac{\gamma}{\beta} + 1\right) \|x_g - x_f\|^2 = \left(1 + \frac{ \left( \gamma - \beta\right)}{2\beta\lambda}\right)\|z - z^+\|^2.
$
Using the above two identities, we have
\begin{align*}
&4\gamma\lambda\big(f(x_g) + g(x_g) - f(x^\ast) - g(x^\ast)\big)\\
&\stackrel{\eqref{eq:fxgupperbound}}{\le}4\gamma\lambda\big(f(x_f) + g(x_g) - f(x^\ast) - g(x^\ast)\big)+ 4\gamma\lambda\dotp{x_g - x_f, \nabla f(x_f)} + \frac{2\gamma\lambda}{\beta}\|x_g - x_f\|^2\\
&\stackrel{\eqref{prop:DRSupper:eq:aux}}{\le} \|z - z^*\|^2 - \|z^+ - z^*\|^2 +\big(2\dotp{z-z^+,z^*-x^*}+ 4\gamma\lambda\dotp{x_g - x_f, \nabla f(x_f)}\big) + \left(1 - \frac{1}{\lambda}\right)\|z^{+} - z\|^2\\
& \hspace{20pt}+\frac{2\gamma\lambda}{\beta}\|x_g - x_f\|^2-4\gamma\lambda S_f(x_f,x^*) \\
&= \|z - z^*\|^2 - \|z^+ - z^*\|^2 + \left(\left(1-\frac{1}{\lambda}\right) \|z - z^+\|^2 + 2\lambda\left(\frac{\gamma}{\beta} + 1\right) \|x_g - x_f\|^2 \right) \\
&\hspace{20pt}+ 2\lambda\|\gamma\nabla f(x_f) - \gamma\nabla f(x^\ast)\|^2-4\gamma\lambda S_f(x_f,x^*) -2\lambda \|\gamma \tnabla g(x_g) - \gamma \tnabla g(x^\ast)\|^2\\
&\stackrel{\eqref{eq:Sflowerbound}}{\le}\|z - z^*\|^2 - \|z^+ - z^*\|^2 +\left(1 + \frac{ \left( \gamma - \beta\right)}{2\beta\lambda}\right)\|z - z^+\|^2+2\gamma\lambda(\gamma-\beta)\|\nabla f(x_f) - \nabla f(x^\ast)\|^2.
\end{align*}
If $\gamma \leq \beta$, we can drop the last term. If $\gamma > \beta$, we apply the upper bound on $S_f(x_f, x)$ in \eqref{prop:sumauxilliaryterms:eq:main} to get
\begin{align*}
2\gamma\lambda\left( \gamma - \beta\right)\|\nabla f(x_f) - \nabla f(x^\ast)\|^2 &\leq \frac{ \left( \gamma - \beta\right)}{2\beta}\left(\|z - z^\ast\|^2 - \|z^+  -z^\ast\|^2 + \left(1 - \frac{1}{\lambda}\right) \|z - z^+\|^2\right),
\end{align*}
and the result follows.
If $\nabla g$ is $({1}/{\beta})$-Lipschitz, the argument is symmetric, so we omit the proof.

\qed\end{proof}

\section{Proofs from Section~\ref{sec:constantrelaxationandbetterates}}\label{app:constantrelaxationandbetterates}

The following two results are well known, but we include some of the proofs for completeness. They will help us tighten the bounds that we develop below.

\begin{lemma}[Extra contraction of derivative operator]\label{lem:extracontraction}
Suppose that $\nabla g$ is $({1}/{\beta})$-Lipschitz, and let $x, y \in \cH$. If $x^+ = \prox_{\gamma g}(x)$ and $y^+ = \prox_{\gamma f}(y)$, then
\begin{align}\label{eq:lipschitzcontraction}
\|\nabla g(x^+) - \nabla g(y^+)\|^2 \leq \frac{1}{\gamma^2 + \beta^2} \|x - y\|^2.
\end{align}
\end{lemma}
\begin{proof}
From the identity $\gamma \nabla g(x^+) = x - x^+$, the contraction property in Proposition~\ref{prop:basicprox}, and the Lipschitz continuity of $\nabla g$ we have
\begin{align*}
\beta^2\|\nabla g(x^+) - \nabla g(y^+)\|^2 \leq \|x^+ - y^+\|^2 && \mathrm{and} && \gamma^2 \|\nabla g(x^+) - \nabla g(y^+)\|^2 \leq \|x - y\|^2 - \|x^+ - y^+\|^2
\end{align*}
Adding both equations and rearranging proves the result.
\qed\end{proof}

The following is a direct corollary of the descent theorem (Theorem~\ref{thm:descent}).

\begin{corollary}[Joint descent theorem]\label{cor:jointdescent}
If $g$ is differentiable and $\nabla g$ is $({1}/{\beta})$-Lipschitz, then for all pairs $x, y \in  \dom(f)$, points $z \cH$, and subgradients $\tnabla f(x) \in \partial f(x)$, we have
\begin{align}\label{cor:jointdescent:eq:main}
f(x) + g(x) &\leq f(y) + g(y) + \dotp{ x- y, \nabla g(z) + \tnabla f(x)} + \frac{1}{2\beta} \|z - x\|^2.
\end{align}
\end{corollary}
\begin{proof}
Inequality \eqref{cor:jointdescent:eq:main} follows from adding  the upper bound
\begin{align*}
g(x) - g(y) &\stackrel{\eqref{eq:lipschitzderivative}}{\leq} g(z) - g(y) + \dotp{ x- z, \nabla g(z)} + \frac{1}{2\beta}\|z - x\|^2 \leq \dotp{x - y, \nabla g(z)} + \frac{1}{2\beta}\|z - x\|^2,
\end{align*}
with the subgradient inequality:
\begin{equation}
f(x) \le f(y) + \dotp{x-y,\tnabla f(x)}.\label{eq:subineq}
\end{equation}
\qed
\end{proof}

The following theorem develops an alternative fundamental inequality to the one in Proposition~\ref{prop:DRSupperfglipschitz}.

\begin{proposition}[Fundamental inequality for differentiable functions]\label{prop:fundamentaldiff}
For all $x \in \dom(f)$,
\begin{align*}
2\gamma&(f(x_f^k) + g(x_f^k) - f(x) - g(x)) + \left(2\gamma \beta - \frac{\gamma^3}{\beta}\right)\| \nabla g(x_g^{k+1}) - \nabla g(x_g^k)\|^2 +  \|x_g^{k+1} - x\|^2 +  \|x_g^{k+1} - x_g^{k} \|^2 \\
&\leq \|x_g^k - x\|^2. \numberthis \label{prop:fundamentaldiff:eq:main}
\end{align*}
\end{proposition}
\begin{proof}
The following identities are straightforward from Lemma~\ref{prop:DRSmainidentity}:
\begin{align*}
x_{g}^{k} - x_{g}^{k+1} = \gamma(\nabla g(x_g^{k+1}) + \tnabla f(x_f^k) ) && \mathrm{and} && x_{f}^k - x_{g}^{k+1} = \gamma(\nabla g(x_g^{k+1}) - \nabla g(x_g^{k})). \numberthis \label{prop:fundamentaldiff:mainidentities}
\end{align*}
Therefore,
\begin{align*}
2\gamma&(f(x_f^k) + g(x_f^k) - f(x) - g(x)) + \left(2\gamma \beta - \frac{\gamma^3}{\beta}\right)\| \nabla g(x_g^{k+1}) - \nabla g(x_g^k)\|^2 \\
&\stackrel{\eqref{cor:jointdescent:eq:main}}{\leq} 2\gamma\dotp{x_f^k - x, \tnabla f(x_f^k) + \nabla g(x_g^{k+1})} + \frac{\gamma}{\beta}\|x_f^{k} - x_g^{k+1}\|^2 + \left(2\gamma \beta - \frac{\gamma^3}{\beta}\right)\| \nabla g(x_g^{k+1}) - \nabla g(x_g^k)\|^2 \\
&\stackrel{\eqref{prop:fundamentaldiff:mainidentities}}{=}2\dotp{x_f^k - x, x_g^k - x_g^{k+1}}  + 2\gamma \beta\| \nabla g(x_g^{k+1}) - \nabla g(x_g^k)\|^2 \\
&= 2\dotp{x_g^{k+1} - x, x_g^k - x_g^{k+1}} + 2\dotp{x_f^k - x_{g}^{k+1}, x_g^k - x_g^{k+1}}  + 2\gamma \beta\| \nabla g(x_g^{k+1}) - \nabla g(x_g^k)\|^2 \\
&\stackrel{\eqref{eq:cosinerule}}{\leq} \|x_g^{k} - x\|^2 - \|x_g^{k+1} - x\|^2 - \|x_g^{k} - x_g^{k+1}\|^2 \\
&\stackrel{\eqref{prop:fundamentaldiff:mainidentities}}{+} 2\gamma\dotp{\nabla g(x_g^{k+1}) - \nabla g(x_g^{k}), x_g^k - x_g^{k+1}} + 2\gamma \beta\| \nabla g(x_g^{k+1}) - \nabla g(x_g^k)\|^2\\
&\stackrel{\eqref{eq:baillon}}{\leq} \|x_g^{k} - x\|^2 - \|x_g^{k+1} - x\|^2 - \|x_g^{k} - x_g^{k+1}\|^2. \numberthis \label{eq:before32}
\end{align*}
Equation~\eqref{prop:fundamentaldiff:eq:main} now follows by rearranging Equation~\eqref{eq:before32}.
\qed\end{proof}

The following proposition uses the fundamental inequality in Proposition~\ref{prop:fundamentaldiff} evaluated at the point $x = x_f^{k-1}$ to construct a monotonic sequence that dominates the objective error. We introduce a factor $\theta \in [0, 1]$ that we will optimize in Lemma~\ref{lem:howtochoosea} in order to maximize the range of $\gamma$ for which the sequence remains monotonic.

\begin{proposition}[Monotonicity]\label{prop:lipschitzmono}
For scalars $\theta \in [0, 1]$ and integers $k \geq 1$, the following bound holds:
\begin{align*}
2\gamma&(f(x_f^k) + g(x_f^k) - f(x^\ast) - g(x^\ast)) + \left(2\gamma \beta - \frac{\gamma^3}{\beta}\right) \| \nabla g(x_g^{k+1}) - \nabla g(x_g^k)\|^2 + \|x_g^{k+1} - x_g^{k} \|^2 \\
&\leq 2\gamma(f(x_f^{k-1}) + g(x_f^{k-1}) - f(x^\ast) - g(x^\ast)) + \theta\gamma^2\|\nabla g(x_g^k) - \nabla g(x_g^{k-1})\|^2 + \frac{(1-\theta)\gamma^2}{\beta^2}\|x_g^k - x_g^{k-1}\|^2 \numberthis \label{prop:lipschitzmono:eq:main}.
\end{align*}
\end{proposition}
\begin{proof}
Plug $x = x_{f}^{k-1}$ into Equation~\eqref{prop:fundamentaldiff:eq:main} and subtract $f(x^\ast) + g(x^\ast)$ from both sides. Equation~\eqref{prop:lipschitzmono:eq:main} follows from the identity
\begin{align*}
x_g^k - x_{f}^{k-1} = \gamma (\nabla g(x_g^{k-1}) - \nabla g(x_g^{k})),
\end{align*}
the bound $\|\nabla g(x_g^k) - \nabla g(x_g^{k-1})\|^2 \leq ({1}/{\beta^2} )\|x_g^k - x_g^{k-1}\|^2$, rearranging, and dropping the positive term $\|x_g^{k+1} - x_f^{k-1}\|^2$.
\qed\end{proof}

We now choose the factor $\theta$ in order to maximize the range of implicit stepsize parameters $\gamma$ for which the sequence constructed in Proposition~\ref{prop:lipschitzmono} remains monotonic.
\begin{lemma}[Maximizing $\gamma$ range]\label{lem:howtochoosea}
Let $\beta > 0$, and let
\begin{align}\label{eq:kappa}
\kappa := \sup\left\{ \frac{\gamma}{\beta} \mid \gamma > 0, \theta \in [0, 1], \theta\gamma^2\leq  \left(2\gamma \beta - \frac{\gamma^3}{\beta}\right) , \; \frac{(1-\theta)\gamma^2}{\beta^2} \leq 1 \right\}.
\end{align}
Then $\kappa$ is the positive root of $x^3 +x^2 - 2x- 1$. Therefore, $(\gamma^\ast, \theta^\ast) = (\kappa\beta, 1 - 1/\kappa^2)$.
\end{lemma}
\begin{proof}
Observe that the constraints on $\theta$ and $\gamma$ are equivalent to following inequalities:
\begin{align*}
1+ \frac{2\gamma}{\beta}  - \frac{\gamma^2}{\beta^2} - \frac{\gamma^3}{\beta^3}  \geq (\theta - 1)\frac{\gamma^2}{\beta^2} +1 \geq 0. \numberthis\label{eq:thetarearrange}
\end{align*}
The left hand side of Equation~\eqref{eq:thetarearrange} is monotonically decreasing in $\gamma$ for all $\gamma \geq \beta$.  Furthermore, if $\gamma = \kappa \beta$, then the left hand side is $0$. Thus, $\gamma^\ast \leq \kappa \beta$. Finally, for every $\gamma \in [0, \kappa \beta]$, the scalar $\theta_\gamma = 1 - \beta^2/\gamma^2$ satisfies $(\theta - 1)(\gamma^2/\beta^2) +1 \geq 0$.  Therefore, $(\gamma^\ast, \theta^\ast) = (\kappa\beta, 1 - 1/\kappa^2)$.
\qed\end{proof}

\begin{remark}\label{rem:kappa}
Throughout the rest of the paper, we will let $\kappa = 1/\sqrt{1-\theta^\ast} \approx  1.24698$ where $\theta^\ast$ is defined in Lemma~\ref{lem:howtochoosea}.  Note that the inequality constraints in Equation~\eqref{eq:kappa} become equalities for the pair $(\gamma^\ast, \theta^\ast)$.
\end{remark}

We will need the following bound in several of the proofs below.
\begin{prop}[Gradient sum bound]
For all $\gamma > 0$
\begin{align}\label{eq:Lipschitzgradientsum}
\sum_{i=0}^\infty \|\nabla g(x_g^i) - \nabla g(x_g^{i+1})\|^2 & \leq \frac{1}{\gamma^2 + \beta^2}\|z^0 - z^\ast\|^2.
\end{align}
\end{prop}
\begin{proof}
From Lemma~\ref{lem:extracontraction} and the Fej\'er type in equality in Equation~\eqref{eq:fejer}:
\begin{align*}
\|\nabla g(x_g^k) - \nabla g(x_g^{k+1})\|^2 &\leq ({1}/({\gamma^2 + \beta^2}))\|z^k - z^{k+1}\|^2 \\
&\leq  ({1}/({\gamma^2 + \beta^2}))\left(\|z^k - z^\ast\|^2 - \|z^{k+1} - z^\ast\|^2\right). \numberthis\label{prop:lipschitzsum:eq:finalalt}
\end{align*}
Therefore, the result follows by summing~\eqref{prop:lipschitzsum:eq:finalalt}.
\qed\end{proof}

The following proposition computes an upper bound of the sum of the sequence in Equation~\eqref{prop:lipschitzmono:eq:main}.
\begin{proposition}[Summability]\label{prop:lipschitzsum}
If $\gamma  < \kappa \beta$, choose $\theta = \theta^\ast$ as in Lemma~\ref{lem:howtochoosea}; otherwise, set $\theta = 1$.  Then
\begin{align*}
\sum_{i=0}^\infty&\left( 2\gamma(f(x_f^k) + g(x_f^k) - f(x^\ast) - g(x^\ast)) + \theta\gamma^2 \| \nabla g(x_g^{k+1}) - \nabla g(x_g^k)\|^2 +\frac{(1-\theta)\gamma^2}{\beta^2}\|x_g^{k+1} - x_g^{k} \|^2\right) \\
&\leq \begin{cases}
\|x_g^0 - x^\ast\|^2, & \mbox{if $\gamma  < \kappa \beta$};\\
\|x_g^0 - x^\ast\|^2 + \frac{1}{\beta^2+ \gamma^2}\left( \frac{\gamma^3}{\beta} - 2\gamma \beta + \gamma^2 - \beta^2\right) \|z^0 - z^\ast\|^2, & \mbox{otherwise.}
\end{cases}
 \numberthis \label{prop:lipschitzsum:eq:main}
\end{align*}
\end{proposition}
\begin{proof}
First note that:
$$-\|x_g^k - x_{g}^{k+1}\|^2 \leq -\beta^2 \|\nabla g(x_g^{k}) - \nabla g(x_g^{k+1})\|^2.$$
In addition, for either choice of $\theta$ we have $ (1-\theta)\gamma^2/\beta^2- 1 \leq 0$.
Thus, from Equation~\eqref{prop:fundamentaldiff:eq:main}
\begin{align*}
2\gamma&(f(x_f^k) + g(x_f^k) - f(x^\ast) - g(x^\ast)) + \theta\gamma^2 \| \nabla g(x_g^{k+1}) - \nabla g(x_g^k)\|^2 +\frac{(1-\theta)\gamma^2}{\beta^2}\|x_g^{k+1} - x_g^{k} \|^2 \\
&\leq \|x_g^k - x^\ast\| - \|x_g^{k+1} - x^\ast\|^2 \\
&+ \left(\frac{\gamma^3}{\beta} - 2\gamma \beta + \theta\gamma^2\right)\|\nabla g(x_g^{k+1}) - \nabla g(x_g^k)\|^2 + \left(\frac{(1-\theta)\gamma^2}{\beta^2}- 1\right)\|x_g^k - x_{g}^{k+1}\|^2 \numberthis \label{prop:lipschitzsum:eq:alternating}\\
&\leq \|x_g^k - x^\ast\| - \|x_g^{k+1} - x^\ast\|^2 + \left(\frac{\gamma^3}{\beta} - 2\gamma \beta + \gamma^2 - \beta^2\right)\|\nabla g(x_g^{k+1}) - \nabla g(x_g^k)\|^2.
\end{align*}
The last line of Equation~\eqref{prop:lipschitzsum:eq:alternating} is negative if, and only if, $\gamma \leq \kappa \beta$. This proves the first bound in Equation~\eqref{prop:lipschitzsum:eq:main}.  The second bound follows from the sum bound in Equation~\eqref{eq:Lipschitzgradientsum}.
\qed\end{proof}

\section{Proofs from Section~\ref{sec:feasibilityniceintersection}}\label{app:feasibilityniceintersection}

In this section, we will vary the implicit stepsize parameter in every iteration. In addition $f$ and $g$ will have separate implicit stepsize parameters. Thus, we augment the $\TPRS$ notation as follows: for all $\gamma_f, \gamma_g > 0$,
\begin{align*}
\TPRS^{\gamma_f, \gamma_g} &:= \refl_{\gamma_f f} \circ \refl_{\gamma_g g}.
\end{align*}

The following optimality conditions are well known. They will be needed in Section~\ref{sec:feasibilityniceintersection} because we vary the implicit stepsize parameter $\gamma$.  See \cite{davis2014convergence,bauschke2011convex} for a proof.

\begin{lemma}[Optimality conditions of $\TPRS $]\label{lem:PRSoptimality}
The set of zeros of $\partial f+ \partial g$ is precisely
\begin{align}
\zer(\partial f + \partial g) &= \{ \prox_{\gamma g}(z) \mid z \in \cH, \TPRS  z = z\}.\numberthis \label{eq:setoptimalitydrs}
\end{align}
That is, if $z^\ast$ is a fixed point of $\TPRS$, then $x^\ast = x_g^\ast = x_f^\ast$ is a solution to Problem~\ref{eq:simplesplit}, and
\begin{align}\label{eq:gradoptimality}
z^\ast - x^\ast = \gamma \tnabla g(x^\ast) \in \gamma \partial g(x^\ast).
\end{align}
Therefore, the set of fixed points of $\TPRS$ is exactly
\begin{align*}
\left\{ x + \gamma w\mid x \in \zer \left( \partial f + \partial g\right),  w\in (-\partial f(x)) \cap \partial g(x)\right\}.
\end{align*}
\end{lemma}

The following propositions study the behavior of $\TPRS^{\gamma_f, \gamma_g}$ as the positive implicit stepsize parameters $\gamma_f $ and $\gamma_g$ vary.
\begin{lemma}[Non expansiveness of PRS operator]
The operator $\TPRS^{\gamma_f, \gamma_g}$ is nonexpansive.
\end{lemma}
\begin{proof}
This is an immediate consequence of the nonexpansiveness of the reflection mapping (See Part~\ref{prop:basicprox:part:nonexpansive} of Proposition~\ref{prop:basicprox}).
\qed\end{proof}

The following lemma will be useful for determining the fixed point set of $\TPRS^{\gamma_f, \gamma_g}$.
\begin{lemma}[Minimizers of weighted squared distance]\label{lem:differentgammadist}
Let  $\rho_1, \rho_2 > 0$, and suppose that $C_f \cap C_g \neq \emptyset$. Then the set of minimizers of $\rho_1 d^2_{C_f} + \rho_2 d^2_{C_g}$ is $C_f \cap C_g$.
\end{lemma}
\begin{proof}
The minimal value is attained whenever $x \in C_f \cap C_g$; otherwise, the sum is nonzero.
\qed\end{proof}

We will now compute the fixed points of $\TPRS^{\gamma_f, \gamma_g}$.
\begin{proposition}[Fixed points of PRS operator]
The set of fixed points of $\TPRS^{\gamma_f, \gamma_g}$ is $C_f \cap C_g$.
\end{proposition}
\begin{proof}
Let $f' = \gamma_f f$ and let $g' = \gamma_g g$.  Then Lemma~\ref{lem:PRSoptimality} combined with Lemma~\ref{lem:differentgammadist} show that the set of fixed points of $\TPRS^{\gamma_f, \gamma_g} = \refl_{f'} \circ \refl_{g'}$ is
\begin{align*}
\{x + \gamma \nabla g'(x) \mid x \in C_f \cap C_g, \nabla g'(x) = -\nabla f'(x)\}.
\end{align*}
However, $\nabla g'(x) = 2\gamma_g(x - P_{C_g}(x))= 0$ for all $x \in C_f \cap C_g$, and so the identity holds.
\qed\end{proof}

Now, we will show that the sequence generated by Equation~\eqref{eq:DRSfeasibilitygamma} is bounded.
\begin{proposition}[Boundedness]
Suppose that $(z^k)$ is generated by the iteration in Equation~\eqref{eq:DRSfeasibilitygamma}. If $(\lambda_k)_{k \geq 0} \subseteq (0, 1]$, then $(\|z^j - x\|^2)_{j \geq 0}$ is monotonically nonincreasing for any $x \in C_f \cap C_g$.
\end{proposition}
\begin{proof}
Because the set of fixed points of $\TPRS^{\gamma_f, \gamma_g}$ does not depend on $\gamma_f$ and $\gamma_g$, the claim follows directly from the Fej\'er-type inequality in Equation~\eqref{eq:fejer}.
\qed\end{proof}

We restate the fundamental inequality here for the readers convenience.

\begin{proposition}[Upper fundamental inequality for feasibility problem]
Suppose that $z\in \cH$ and $z^+ = (\TPRS^{\gamma_f, \gamma_g})_{\lambda}(z)$. Then for all $x^\ast\in C_f\cap C_g$,
\begin{align}
8\lambda(\gamma_{f} d^2_{C_f}(x_f) + \gamma_{g}d^2_{C_g}(x_g))  &\leq \|z -  x^\ast\|^2 -  \|z^{+} - x^\ast\|^2 + \left(1 - \frac{1}{\lambda} \right)\|z^{+} - z\|^2.
\end{align}
\end{proposition}
\begin{proof}[Proof of Proposition~\ref{prop:factsaboutdistancesquared}]
This follows directly from the upper fundamental inequality in Proposition~\ref{prop:DRSupper} (with $\mu_f = \mu_g = 0$, and $\gamma = 1$), applied to the functions $f' = \gamma_f f$ and $g' = \gamma_g g$.  Indeed,  the gradients $\gamma_{f}\nabla d_{C_f}^2$ and $\gamma_{g}\nabla d_{C_g}^2$ are $2\gamma_{f}$ and $2\gamma_{g}$-Lipschitz ($\beta_{f'} = {1}/({2\gamma_{f}})$ and $\beta_{g'} = {1}/({2\gamma_{g}})$). Furthermore, if $S_{g'}$ and $S_{f'}$ are defined as in Equation~\eqref{eq:snotation}, then
\begin{align}
S_{g'}(x_{g'}, x^\ast) = \frac{1}{4\gamma_g} \|\gamma_{g}\nabla d_{C_g}^2(x_g) - \gamma_{g}\nabla d_{C_g}^2(x^\ast)\|^2  =\frac{\gamma_{g}}{4}\|2(x_g - P_{C_g}(x_g))\|^2 = \gamma_{g}d_{C_g}^2(x_g),
\end{align}
and by the same argument, $S_{f'}(x_{f'}, x^\ast) = \gamma_{f} d_{C_f}^2(x_f)$.  To summarize, we have
\begin{align}
\gamma_{f} d^2_{C_f}(x_f) + \gamma_{g}d^2_{C_g}(x_g) + S_{f'}(x_{f'}, x^\ast) + S_{g'}(x_{g'}, x^\ast)  &= 2\gamma_{f}d^2_{C_f}(x_f) + 2\gamma_{g}d^2_{C_g}(x_g).
\end{align}
Therefore, the inequality follows because $d_{C_g}^2(x^\ast) = d_{C_f}^2(x^\ast) = 0$.
\qed\end{proof}

\section{Extension of results of Section~\ref{sec:feasibilityniceintersection} to multiple sets}~\label{app:feasibilitymultiplesets}


The concept of (bounded) linear regularity is defined for any finite number of sets. The following theorem shows that (bounded) linear regularity of a collection of sets is equivalent to the (bounded) linear regularity of a certain pair of sets in a product space.  For convenience we set
\begin{align}
D &:= \{(x, \cdots, x) \mid x \in \cH\} \subseteq \cH^m,
\end{align}
and endow $\cH^m$ with the canonical norm: $\|(x_1, \cdots, x_m)\|^2 = ({1}/{m})\sum_{i=1}^m \|x_i\|^2$. We will use the boldface notation $\vx \in \cH^m$ for an arbitrary vector in $\cH^m$. Finally, for any $\vx \in \cH^m$, we will write $\vx_{j}$ for the $j$th component of $\vx$, which is an element of $\cH$.
\begin{theorem}[(Bounded) linear regularity in product spaces] \label{thm:boundedlinearlyregularproduct}
Suppose that $C_1, \cdots, C_m$ are closed convex subsets of $\cH$ with nonempty intersection.  Then $\{C_1, \cdots, C_m\}$ is boundedly linearly regular or linearly regular, if, and only if, $\{C_1 \times \cdots \times C_m, D\}$ has the same property in $\cH^m$ with the canonical norm.  In particular, if $\{C_1, \cdots, C_m\}$ is $\mu_\rho$-(boundedly) linearly regular on the ball $B(0, \rho)$, then $\{C_1 \times \cdots \times C_n, D\}$ is $\sqrt{(1+4m\mu_\rho^2)}$-(boundedly) on the ball $B(\mathbf{0}, \rho)$, and
\begin{align}
d_{(C_1\times \cdots \times C_m) \cap D}(\vx) &\leq \sqrt{(1+ 4m\mu_\rho^2)}\max\{d_{C_1\times \cdots \times C_m}(\vx), d_{D}(\vx)\}.
\end{align}
\end{theorem}
\begin{proof}
See \cite[Theorem 3.12]{deutsch2008rate}.
\qed\end{proof}

In this section we model the feasibility problem of the $m$ sets $\{C_{1}, \cdots, C_m\}$ using the following two objective functions on the product space $\cH^m$:
\begin{align*}
f(x_1, \cdots, x_m) = \sum_{i=1}^m d_{C_i}^2(x_i) && \mathrm{and} &&  g(x_1, \cdots, x_n) = d_{D}^2(x_1, \cdots, x_n).
\end{align*}
In the space $\cH^m$, the proximal operators of $f$ and $g$ have the following form:
\begin{align*}
\prox_{\gamma f} (\vx) = \left(\frac{1}{2\gamma + 1} x_i  + \frac{2\gamma}{2\gamma + 1} P_{C_j}x_j\right)_{j = 1}^m && \mathrm{and} && \prox_{\gamma g}(\vx)  = \left(\frac{1}{2\gamma + 1}x_j + \frac{2\gamma }{(2\gamma + 1)m}\sum_{i=1}^m x_i\right)_{j=1}^m.
\end{align*}
We apply the iteration in Equation~\eqref{eq:DRSfeasibilitygamma} with these identities to get the following parallel algorithm: given implicit stepsize parameters $(\gamma_{f, j})_{ j \geq 0}$ and $(\gamma_{g, j})_{j \geq 0}$, relaxation parameters $(\lambda_j)_{j \geq 0} \subseteq (0, 1]$, and an initial point $\vz^0 \in \cH^m$, for all $k \geq 0$, define
\begin{align*}
&\overline{z}^k = \frac{1}{m}\sum_{i=1}^m z_i^k;\\
&\begin{cases}
x_{g,i}^k = ({1}/({2\gamma_{g, k} + 1}))z^k_i+ ({2\gamma_{g, k}}/({2\gamma_{g, k} + 1}))\overline{z}^k; & \mbox{For}  \\
x_{f,i}^k = ({1}/({2\gamma_{f, k}+1}))(2x_{g, i}^k - z_i^k) + ({2\gamma_{f, k}}/({2\gamma_{f, k} + 1})) P_{C_i}(2x_{g, i}^k - z_i^k);  & \boxed{i = 1, \cdots, m} \\
z_{i}^{k+1} = z_i^k + 2\lambda_k (x_{f, i}^k - x_{g, i}^k); & \mbox{in parallel.}
\end{cases} \numberthis \label{eq:PPSdistance}
\end{align*}
Note that the algorithm in Equation~\eqref{eq:PPSdistance} is related to the general algorithm in \cite[Section 8.3]{bauschke1996projectionthesis}. One of the main differences between these two algorithms is that the projection operators are not necessarily evaluated at same point in each iteration ($(2\vx_g^k - \vz^k) \notin D$). By changing the metric of the underlying space, e.g. to $\|(x_1, \cdots, x_m)\|^2 = \sum_{i=1}^m w_i \|x_i\|^2$ where $w_i > 0$ are arbitrary weights, we can perform a weighted average of all the projections.  In addition, we can assign each set $C_i$ a different implicit stepsize parameter at each iteration. For simplicity we do not pursue these extensions here.

The following theorem deduces the linear convergence of the iteration in Equation~\eqref{eq:PPSdistance}.
\begin{theorem}[Linear convergence: Feasibility for multiple sets]\label{thm:PPSfeasibility}
Suppose that $(\vz^j)_{j \geq 0}$ is generated by the iteration in Equation~\eqref{eq:PPSdistance}, and suppose that $\{C_1, \cdots, C_m\}$ is (boundedly) linearly regular.  Let $\rho > 0$ and $\mu_\rho > 0$ be such that $(\vz^j)_{ j \geq 0} \subseteq B(\mathbf{0}, \rho)$ and the inequality
\begin{align}
d_{C_1 \cap \cdots \cap C_m}(x) &\leq \mu_\rho \max\{d_{C_1}(x),\cdots,  d_{C_m}(x)\}
\end{align}
holds for all $x \in B(0, \rho)$. Then $(\vz^j)_{j \geq 0}$ satisfies the following relation: for all $k \geq 0$,
\begin{align}
d_{(C_1\times \cdots \times C_m) \cap D}(\vz^{k+1}) &\leq C(\gamma_{f, k}, \gamma_{g, k}, \lambda_k, \mu_\rho)  \times d_{(C_1\times \cdots \times C_m) \cap D}(\vz^k)
\end{align}
where
\begin{align*}
C(\gamma_{f, k}, \gamma_{g, k}, \lambda_k, \mu_\rho) &:= \left(1-\frac{4\lambda_k\min\{{\gamma_{g, k}}/{(2\gamma_{g, k} + 1)^2}, {\gamma_{f, k}}/{(2\gamma_{f, k} + 1)^2}\}}{(1+4m\mu_\rho^2)\max\{{16\gamma_{g, k}^2}/{(2\gamma_{g, k}+1)^2}, 1\}} \right)^{{1}/{2}}.
\end{align*}
In particular, if $\overline{C} = \sup_{j \geq 0}C(\gamma_{f, k}, \gamma_{g, k}, \lambda_k, \mu_\rho)  < 1$, then $(\vz^j)_{j \geq 0}$ converges linearly to a point in $(C_1\times \cdots \times C_m) \cap D$ with rate $\overline{C}$, and
\begin{align}
\|\vz^k - \vz^\ast\| &\leq 2d_{(C_1\times \cdots \times C_m) \cap D}(\vz^0) \prod_{i=0}^k C(\gamma_{f, i}, \gamma_{g, i}, \lambda_i, \mu_\rho).
\end{align}
\end{theorem}
\begin{proof}
This theorem is a direct corollary of Theorem~\ref{thm:linearfeasibility} except that Theorem~\ref{thm:boundedlinearlyregularproduct} is used to calculate the (bounded) linear regularity constant.
\qed\end{proof}

Finally we derive the following analogue of Corollary~\ref{cor:AP}.
\begin{corollary}[Convergence of MAP: Multiple sets]\label{cor:APm}
Let $(\vz^j)_{j \geq 0}$ be generated by the iteration in Equation~\eqref{eq:PPSdistance} with $\gamma_{f, k} \equiv \gamma_{g, k} \equiv {1}/{2}$ and $\lambda_k \equiv1$.  Define $x^k :=(P_D\vz^k)_1$. Then for all $k \geq 0$,
\begin{align}\label{eq:AAP}
x^{k+1} = \frac{1}{m}\sum_{i=1}^m P_{C_i}(x^k).
\end{align}
Thus, Averaged MAP is a special case of  PRS.   Consequently, under the assumptions of Theorem~\ref{thm:PPSfeasibility},  $x^k$ converges linearly to a point in the intersection $C_{1} \cap \cdots \cap C_m$ with rate $\left(1 - {1}/({1+4m\mu^2})\right)^{{1}/{2}}$.
\end{corollary}
\begin{proof}
Equation~\eqref{eq:AAP} follows because $\refl_{\gamma g} = P_D$ and $\refl_{\gamma f} = P_{C_1\times \cdots \times C_m}$.  In addition,  by the nonexpansiveness of $P_{D}$ we have
\begin{align*}
\|x^k - z^\ast\|^2 = \|(P_D\vz^k)_{1} - \vz_1^\ast\|^2 = \frac{1}{m}\sum_{i=1}^k \|(P_D\vz^k)_i - \vz_i^\ast\|^2 \leq \|\vz^k - \vz^\ast\|^2.
\end{align*}
By Corollary~\ref{cor:AP} and Theorem~\ref{thm:boundedlinearlyregularproduct}, the sequence $(\vz^j)_{j \geq 0}$ converges linearly with rate $\left(1 - {1}/({1+4m\mu^2})\right)^{{1}/{2}}$. Thus, the rate for $(x^k)_{k \geq 0}$ follows from the rate for $(\vz^j)_{j \geq 0}$.
\qed\end{proof}

\section{Consequences of linear convergence of ADMM}\label{app:linearconvergenceimplies}

The following proposition is a translation of Proposition~\ref{prop:linearconvergenceimplies} to the ADMM setting.
\begin{proposition}[Consequences of linear convergence of ADMM]\label{prop:linearconvergenceimpliesADMM}
Let $(C_j)_{j \geq0} \subseteq [0, 1]$ be a positive scalar sequence, and suppose that for all $k \geq0$,
\begin{align}\label{ADMM:prop:linearconvergenceimplies:eq:main}
\|z^{k+1} - z^\ast\| \leq C_k\|z^k - z^\ast\|.
\end{align}
Fix $k \geq 1$. Then
\begin{align*}
 \|w_{d_g}^k - w^\ast\|^2 + \gamma^2\| By^k - By^\ast\|^2 \leq \|z^0 - z^\ast\|^2\prod_{i = 0}^{k-1} C_i^2; \\
 \|w_{d_f}^k - w^\ast\|^2 + \gamma^2\|Ax^k - Ax^\ast\|^2 \leq \|z^0 - z^\ast\|^2\prod_{i = 0}^{k-1} C_i^2.
 \end{align*}
If $\lambda <1$, then the FPR rate holds: $$\|(\TPRS)_{\lambda}z^k- z^k\| \leq \sqrt{\frac{\lambda}{1-\lambda}}\|z^0 - z^\ast\|\prod_{i = 0}^{k-1} C_i.$$ Consequently, the following convergence rates for constraint violations and objective errors hold:$$\|Ax^k + By^k - b\|^2 \leq \frac{\|z^0 - z^\ast\|^2}{\gamma^2}\prod_{i = 0}^{k-1} C_i^2,$$  and
\begin{align*}
\frac{-\|z^0 - z^\ast\|\|w^\ast\|}{\gamma}\prod_{i = 0}^{k-1} C_i \leq f(x^k) + g(y^k) - f(x^\ast) - g(y^\ast) \leq   \frac{\left(\|z^0 - z^\ast\| + \|w^\ast\|\right)\|z^0 - z^\ast\|}{\gamma} \prod_{i = 0}^{k-1} C_i.
\end{align*}
\end{proposition}
\begin{proof}
The convergence rates for the dual variables, primal variables, and  FPR follow from Proposition~\ref{prop:linearconvergenceimplies} and the identities in Table~\ref{table:ADMMsubgradients}.

Now fix $k \geq 1$, and let $z_\lambda = (\TPRS)_{\lambda}z^k$ for all $\lambda \in [0,1]$. The convergence rate for the constraint violation follows from the identity $z_\lambda - z^k = -2\gamma \lambda( Ax^{k} + By^k - b)$ (Equation~\eqref{eq:ADMMfeasibilityFPR}) and the FPR convergence rate:
\begin{align*}
\|Ax^k + By^k - b\|^2 = \inf_{\lambda \in [0, 1]} \frac{\|z_\lambda - z^k\|^2}{4\gamma^2\lambda^2}= \inf_{\lambda \in [0, 1]}\frac{\|z^0 - z^\ast\|^2}{4\gamma^2\lambda(1-\lambda)}\prod_{i = 0}^{k-1} C_i^2 =  \frac{\|z^0 - z^\ast\|^2}{\gamma^2}\prod_{i = 0}^{k-1} C_i^2
\end{align*}
The lower bound on the objective error follows from the fundamental lower inequality in Proposition~\ref{prop:ADMMlower} and the constraint violations rate:
\begin{align*}
f(x^k) + g(y^k) - f(x^\ast) - g(y^\ast)\stackrel{\eqref{prop:ADMMlower:eq:main}}{\geq} \dotp{Ax^k + By^k - b, w^\ast} \geq -\|Ax^k + By^k - b\|\| w^\ast\| \geq \frac{-\|z^0 - z^\ast\|\|w^\ast\|}{\gamma}\prod_{i = 0}^{k-1} C_i
\end{align*}
The upper bound on the objective error follows from Proposition~\ref{prop:ADMMupper}, the FPR rate, the bound $\|z_{\lambda} - z^\ast\|^2 \leq \|z^k - z^\ast\|$ (Equation~\eqref{eq:fejer}), the monotonicity of the sequence $(\|z^{j} - z^\ast\|)_{j \geq 0}$ (Part~\ref{fact:averagedconvergence:eq:mono} of Fact~\ref{fact:averagedconvergence}), and the following inequalities:
\begin{align*}
 f(x^k) + g(y^k) &- f(x^\ast) - g(y^\ast)\\
&\stackrel{\eqref{prop:ADMMupper:eq:main}}{\leq} \inf_{\lambda \in [0, 1]} \frac{1}{4\gamma\lambda}\left(\|z^k - (z^\ast - w^\ast)\|^2 - \|z_{\lambda} - (z^\ast - w^\ast)\|^2  + \left(1- \frac{1}{\lambda} \right) \|z^{k} - z_\lambda\|^2\right) \\\
&\stackrel{\eqref{eq:cosinerule}}{\leq}\inf_{\lambda \in [0, 1]}  \frac{1}{4\gamma\lambda}\left(2\dotp{ z_\lambda - (z^\ast - w^\ast), z^k - z_\lambda} + 2\left(1- \frac{1}{2\lambda} \right) \|z^{k} - z_\lambda \|^2 \right) \\
&\leq \frac{\left(\|z_{1/2} - z^\ast\| + \|w^\ast\|\right)\|z_{1/2} - z^k\|}{\gamma}  \\
&\leq \frac{\left(\|z^0 - z^\ast\| + \|w^\ast\|\right)\|z^0 - z^\ast\|}{\gamma } \prod_{i = 0}^{k-1} C_i.
\end{align*}
\qed\end{proof}

\section{Applications to conic programming}\label{app:conicprogramming}

In this section we borrow the setting of \cite{o2013operator}. The goal of linear (LP) and semidefinite (SDP) programming is to minimize a linear function subject to linear and matrix semidefinite constraints, respectively.  Thus, in this section we study the following generic primal-dual pair problem
\begin{align*}
\Min_{x \in \vR^n} &\; c^T x && & \Max_{y \in \vR^m} &\;  -b^Ty \\
\mbox{subject to} &\; Ax + s = b && &\mbox{subject to} &\;  - A^T y + r = c  \\
&\; (x, s) \in \vR^n \times \cK && &&  (r, y) \in \{0\}^n \times \cK^\ast \numberthis\label{eq:conicprogramming}
\end{align*}
where $c \in \vR^n$, $b, s \in \vR^m$,  $A : \vR^n \rightarrow \vR^m$ is a linear map, $\cK \subseteq \vR^m$ is a closed convex cone, and $\cK^\ast \subseteq \vR^m$ is the dual cone to $\cK$. In linear programming $\cK$, is the positive orthant $\cK = \vR^n_+$, and for semidefinite programming, $\cK$ is the cone of symmetric, positive semidefinite matrices.

In \cite{o2013operator}, both optimization problems in Equation~\eqref{eq:conicprogramming} are combined into a single feasibility problem. To this end we introduce slack variables $\tau, \kappa \in \vR_+$, and the vectors and matrix
\begin{align*}
u =\begin{bmatrix} x \\ y \\ \tau \end{bmatrix} \in \vR^{n + m + 1}, \quad  v = \begin{bmatrix} r \\ s \\ \kappa\end{bmatrix} \in \vR^{n + m + 1}, \quad Q = \begin{bmatrix} 0 & A^T & c \\ - A & 0 & b \\ -c^T &- b^T  &0 \end{bmatrix} \in \vR^{(n + m + 1)\times (n+ m + 1)}.
\end{align*}
In addition, we let $\cC = \vR^n \times \cK^\ast \times \vR_+$ and $\cC^\ast = \{0\} \times \cK \times \vR_+$.  With this notation the goal of the \emph{homogeneous self dual embedding} problem is to find $(u, v) \in \vR^{n+m + 1}$ such that $Qu = v$ and $(u, v) \in \cC \times \cC^\ast$.

Throughout this section we denote
\begin{align}
C_f = \cC \times \cC^\ast && \mathrm{and} &&  C_g = \{(u, v) \in \vR^{n + m + 1} \times \vR^{n+m+1} \mid Qu = v\}.
\end{align}
Our goal is to find a point in the intersection $C_f \cap C_g$. A remarkable trichotomy was derived in \cite{ye1994nl}: Suppose $(u, v) \in C_f \cap C_g$, then
\begin{enumerate}
\item If $\tau > 0$ and $\kappa = 0$, then $(x/\tau, y/\tau, s/ \tau)$ is a primal dual solution of \label{eq:conicprogramming}.
\item If $\tau = 0$ and $\kappa > 0$, then $c^T x + b^Ty < 0$. The case $b^T y < 0$ is a certificate of primal infeasibility, and the case $c^Tx < 0$ is a certificate of dual in feasibility.
\item If $\tau = \kappa = 0$, then nothing can be concluded about Equation~\eqref{eq:conicprogramming}. However, if there exists a point $(u', v') \in C_f \cap C_g$ for which $\tau' + \kappa' \neq 0$, then we can choose an initial point $z^0 \in \vR^{n + m + 1}$ such that DRS applied with $f= \iota_{C_f}$ and $g = \iota_{C_g}$ converges to a point $(u', v') \in C_f\cap C_g$ with $\kappa' + \tau' \neq 0$~\cite{o2013operator}.
\end{enumerate}

\subsection{Linear programming}\label{sec:linearprogramming}
Let us now examine the structure of the sets $C_f$ and $C_g$.  For linear programming problems, $C_f = \vR^n \times \vR_+^m \times \vR_+ \times \{0\} \times \vR_+^m \times \vR_+$ is a polyhedron, i.e. the intersection of finitely many half planes, and $C_g$ is a linear subspace. In finite dimensional spaces the pair $\{C_f, C_g\}$ is linearly regular in the sense of Definition~\ref{defi:linearregularity} \cite[Remark 5.7.3]{bauschke1996projectionthesis}.

We have four different algorithms that we can apply to find a point in $C_f \cap C_g$. The first two are the non parallelized versions of DRS which correspond to function pairs
\begin{align}\label{eq:linearprogpairnoparallel}
(f = \iota_{C_f}, g = \iota_{C_g}) && \mathrm{and} &&  (f = d_{C_f}^2, g = d_{C_g}^2).
\end{align}
Theorem~\ref{thm:linearfeasibility} shows that relaxed PRS applied to the second pair (Equation~\eqref{eq:DRSfeasibilitygamma}) linearly convergence to a point in the intersection $C_f \cap C_g$.  Linear convergence of DRS applied to the first pair was shown in \cite{bauschke2014linear}.

The projection onto $C_f$ is simple, and so the main computational bottleneck of the algorithm is to project onto $C_g$.  There are various tricks that can be employed to speed this step up \cite{o2013operator}, but in some cases it is desirable to break up the linear equations into several sets $C_g = C_{g_1} \cap \cdots \cap C_{g_r}$ where $C_{g_i} \subseteq \vR^{n + m + 1}$ each encode a small number of linear constraints.

The collection $\{C_f, C_{g_1}, \cdots, C_{g_r}\}$ is linearly regular by \cite[Remark 5.7.3]{bauschke1996projectionthesis}, so we can apply Theorem~\ref{thm:boundedlinearlyregularproduct} to show that
$\{C_f \times C_{g_1} \times \cdots \times C_{g_r}, D\}$ is linearly regular where $D \subseteq \vR^{(r+1)(n + m + 1)}$ is the ``diagonal set" of Appendix~\ref{app:feasibilitymultiplesets}.  Thus, we can apply DRS or relaxed PRS to either of the following pairs:
\begin{align}\label{eq:linearprogpairparallel}
(f = \iota_{C_f\times C_{g_1} \times \cdots \times C_{g_r}}, g = \iota_{D}) && \mathrm{and} &&  (f = d_{C_f\times C_{g_1} \times \cdots \times C_{g_r}}^2, g = d_{D}^2).
\end{align}
We can deduce linear convergence of the first pair using  \cite{bauschke2014linear} and of the second by Theorem~\ref{thm:PPSfeasibility}.

In general, the pairs in Equation~\eqref{eq:linearprogpairnoparallel} and~\eqref{eq:linearprogpairparallel} may not perform the same in practice.  Thus, we cannot make any prediction about the practical performances of the methods. We can only point to our arguments in Section~\ref{sec:feasibilitygeneralconvergence} that seem to indicate a better performance of the indicator function pair in problems that are badly conditioned.

\subsection{Semidefinite programming}
For semidefinite programming, $\cK$ is the cone of positive semidefinite matrices. Note that $\cK^\ast = \cK$, i.e. $\cK$ is self dual \cite[Example 6.25]{bauschke2011convex}.  In general, the pair $\{C_f, C_g\}$ is not necessarily (boundedly) linearly regular. The main condition to check is whether the relative interior of $C_f$ intersects the subspace $C_g$ \cite[Theorem 5.6.2]{bauschke1996projectionthesis}. In fact, the relative interior of $\cK$ in $\vR^{m}$ is the set of all strictly positive definite matrices, i.e. the set of full rank positive definite matrices. Many problems of interest in semidefinite programming arise from the \emph{lifting} of a non convex problem and desire \emph{low rank} solutions of the associated SDP \cite{goemans1995improved}. Thus, we do not expect the relative interior of $C_f$ to intersect $C_g$ for every SDP.

In terms of algorithm choice, we have at least four options to model the feasibility problem (See Equations~\eqref{eq:linearprogpairnoparallel} and~\eqref{eq:linearprogpairparallel}). In particular, when the linear constraints are difficult to solve in unison, we can break them into smaller pieces and solve them exactly. However, the main computational bottleneck of semidefinite programming is the projection onto the semidefinite cone.  Unfortunately, there seems to be no way to lighten the cost of this projection.

\cut{The convergence rates for relaxed PRS applied to the feasibility problem are linear whenever the relative interior of $C_f$ intersects $C_g$.  In terms of $\cK$ this condition requires that there is a full rank strictly positive definite primal dual pair $(x, y) \in \cK \times \cK$.  Finally, because we usually do not expect full rank solutions to SDPs,}We refer the reader to Section~\ref{sec:feasibilitygeneralconvergence} and Equations~\eqref{eq:feasibilitybounddistancenonergodic} and~\eqref{eq:drsergodicdistancebound} which show the worst case feasibility convergence rates.

\end{document}